\documentclass[a4paper,12pt]{amsart}

\usepackage{amssymb, mathtools}
\usepackage[normalem]{ulem}

\allowdisplaybreaks

\newtheorem{thm}{Theorem}
\newtheorem{cor}[thm]{Corollary}
\newtheorem{lem}[thm]{Lemma}
\newtheorem*{thm*}{Theorem}
\newtheorem{prop}[thm]{Proposition}

\theoremstyle{definition}
\newtheorem{dfn}[thm]{Definition}

\theoremstyle{remark}
\newtheorem{rmk}[thm]{Remark}

\newtheorem{example}[thm]{Example}

\numberwithin{equation}{section}
\numberwithin{thm}{section}

\newcommand{\NN}{\mathbb{N}}
\newcommand{\ZZ}{\mathbb{Z}}

\newcommand{\CC}{\mathbb{C}}

\newcommand{\FF}{\mathbb{F}}
\newcommand{\PP}{\mathbb{P}}

\def\Tt{\mathcal{T}}
\def\Kk{\mathcal{K}}
\def\Ii{\mathcal{I}}

\def\Ii{\mathcal{I}}

\newcommand*{\nb}{\nobreakdash}
\newcommand{\Cst}{\mathrm{C}^*}
\newcommand{\maxt}{\mathrm{max}}
\newcommand{\mint}{\mathrm{min}}
\newcommand{\clsp}{\overline{\operatorname{span}}}
\newcommand{\lsp}{\operatorname{span}}
\newcommand{\id}{\operatorname{id}}

\RequirePackage{color}

\begin{document}

\author[Astrid an Huef]{Astrid an Huef}
\author[Brita Nucinkis]{Brita Nucinkis}
\author[Camila F. Sehnem]{Camila F. Sehnem}
\author[Dilian Yang]{Dilian Yang}

\address{School of Mathematics and Statistics, Victoria University of Wellington, P.O. Box 600, Wellington 6140, New Zealand.}
\email{astrid.anhuef@vuw.ac.nz}

\address{Department of Mathematics, Royal Holloway, University of London, Egham, TW20 0EX, UK.}
\email{brita.nucinkis@rhul.ac.uk}

\address{School of Mathematics and Statistics, Victoria University of Wellington, P.O. Box 600, Wellington 6140, New Zealand.}
\email{camila.sehnem@vuw.ac.nz}

\address{Department of Mathematics and Statistics, University of Windsor, 401 Sunset Avenue, Windsor, Ontario N9B 3P4, Canada.}
\email{dyang@uwindsor.ca}

\title[Nuclearity of semigroup $\mathbf{\Cst}$-algebras]{Nuclearity of semigroup $\mathbf{\Cst}$-algebras}

\keywords{Toeplitz algebra, semigroup, weak quasi-lattice, nuclearity, amenability, Baumslag--Solitar group}
\subjclass[2010]{46L05}

\begin{abstract} We study the semigroup $\Cst$\nb-algebra of a positive cone $P$ of a weakly quasi-lattice ordered group. That is, $P$ is a subsemigroup of a discrete group~$G$ with $P\cap P^{-1}=\{e\}$ and such that any two elements of $P$ with a common upper bound in~$P$ also have a least upper bound. We find sufficient conditions for the semigroup $\Cst$\nb-algebra of~$P$ to be nuclear. These conditions involve the idea of a generalised length function, called a ``controlled map'',  into an amenable group. Here we give a new definition of a controlled map and discuss examples from different sources. We apply our main result to establish nuclearity for semigroup $\Cst$\nb-algebras of a class of one-relator semigroups, motivated by a recent work of Li, Omland and Spielberg. This includes all the Baumslag--Solitar semigroups. We also analyse semidirect products of weakly quasi-lattice ordered groups and use our theorem in examples to prove nuclearity of the semigroup $\Cst$\nb-algebra. Moreover, we prove that the graph product of weak quasi-lattices is again a weak quasi-lattice, and show that the corresponding semigroup $\Cst$\nb-algebra is nuclear when the underlying groups are amenable.
\end{abstract}

\date{October 10, 2019; with minor revisions September 4, 2020}
\thanks{This research was initiated during  the project-oriented workshop ``Women in Operator Algebras'' funded and hosted by the Banff International Research Station. A.~an Huef was supported by  the  Marsden Fund of the Royal Society of New Zealand (grant number 18-VUW-056). Part of this work was completed while C. F. Sehnem was a post-doctoral fellow at Universidade Federal de Santa Catarina supported by CAPES/PrInt (Brazil), through 88887.370650/2019-00.
D.~Yang was partially supported by an NSERC Discovery Grant (grant number 808235). }

\maketitle

\section{Introduction} $\Cst$\nb-algebras associated to semigroups provide a rich class of examples of $\Cst$\nb-alge\-bras. There are interesting connections of these algebras with number theory \cite{CDL, LR2, Li-TAMS} that were sparked by Cuntz's study of 
the $\Cst$\nb-algebra of an $ax+b$\nb-semigroup over the natural numbers~\cite{Cuntz}. See also \cite{BruLi} for recent developments in this subject. Further interesting examples include $\Cst$\nb-algebras of positive cones of right-angled Artin groups \cite{CL1, CL2, ELR} and of Baumslag--Solitar groups  \cite{CaHR, SBS}, $\Cst$\nb-algebras of self-similar actions \cite{BRRW, LY}
 and least common multiple (LCM) semigroups \cite{BLS, BLS2, Starling}.

In \cite{Li-JFA}, Li introduced semigroup $\Cst$\nb-algebras that generalise the ones of Nica attached to quasi-lattice ordered groups. Li's construction takes into account the structure of constructible right ideals of the underlying semigroup. These are all the principal ideals together with the emptyset if~$P$ is a left-cancellative and right LCM semigroup. In this case, the full semigroup $\Cst$\nb-algebra $\Cst(P)$ is universal for a class of isometric representations of~$P$ satisfying a relation called Nica covariance. The reduced semigroup $\Cst$\nb-algebra $\Cst_r(P)$ is then generated by the left-regular representation of $P$ on $\ell^2(P)$.

In this paper, we study the $\Cst$\nb-algebras of semigroups that are positive cones of weakly quasi-lattice ordered groups. These sit between quasi-lattice ordered groups and right LCM semigroups. Following Nica, we carry the group in our notation, and write $\Cst(G,P)$ for the semigroup $\Cst$\nb-algebra associated to the positive cone~$P$ of the weak quasi-lattice order $(G,P)$. The group is crucial in our proofs and this notation helps to make our assumptions clear throughout the text, but we observe that $\Cst(G,P)$ depends only on~$P$. 

As in \cite{N}, we say that $(G,P)$ is amenable if the left-regular representation induces an isomorphism between the full and reduced semigroup $\Cst$\nb-algebras. If $G$ is amenable as a group, then $(G,P)$ is amenable. In fact, $\Cst(G,P)$ is nuclear in this case \cite[Corollary~6.45]{Li-book}. In general, it is not known whether amenability of the (weak) quasi-lattice order $(G,P)$ implies nuclearity of its $\Cst$\nb-algebra. In addition, nuclearity of $\Cst(G,P)$ implies that of its boundary quotient semigroup $\Cst$\nb-algebra. Since this latter is quite often purely infinite and simple \cite{CL2,LOS}, nuclearity results for semigroup $\Cst$\nb-algebras combined with \cite{Tu1999} may give new examples of $\Cst$\nb-algebras that are classifiable by $K$\nb-theory. Our main purpose in this article is to analyse nuclearity of $\Cst(G,P)$ and to contribute to the existing literature with new examples.

A sufficient condition for amenability of quasi-lattice orders was established by Laca and Raeburn in \cite[Proposition~6.6]{LR}. This condition involves the existence of a generalised length function, now called a ``controlled map'', from~$G$ into an amenable group. The controlled maps from \cite{LR} have proved very useful. They were used in \cite{CL1} to show that any graph product of a family of quasi-lattice orders in which the
underlying groups are amenable is again an amenable quasi-lattice order. The existence of a controlled map into an amenable group implies that the partial action of~$G$ on any closed invariant subset of the Nica spectrum is amenable \cite[Theorem~4.7]{CL2} and that $\Cst(G,P)$ is nuclear \cite[Corollary~8.3]{Li}. Spielberg used ideas of a similar flavour to give sufficient conditions for nuclearity of the $\Cst$\nb-algebra associated to a finitely aligned category of paths (see \cite[Definition~9.5 and Theorem~9.8]{SCP}).

A controlled map as originally introduced by Laca and Raeburn in \cite{LR} is an order-preserving group homomorphism~$\phi$ between quasi-lattice ordered groups that preserves the least upper bound of any two elements of~$P$ when it exists, and such that $\ker\phi\cap P$ is trivial. In \cite{aHRT}, an Huef, Raeburn and Tolich introduced a version of a controlled map whose restriction to the positive cone allows a nontrivial kernel, and used it to prove an amenability result along the lines of \cite{LR}. A key example of a controlled map in the sense of~\cite{aHRT} comes from the height map on a Baumslag--Solitar group. Recall that, for $c,d\geq 1$, the Baumslag--Solitar group $\mathrm{BS}(c,d)$ is defined by 
\[\mathrm{BS}(c,d)=\langle a, b: ab^c=b^{d}a\rangle.\] 
 Spielberg proved that $(\mathrm{BS}(c,d),\mathrm{BS}(c,d)^+)$ is quasi-lattice ordered, where $\mathrm{BS}(c,d)^+$ is the unital subsemigroup of $\mathrm{BS}(c,d)$ generated by~$a$ and~$b$ \cite{SBS}. The ``height map'' from $\mathrm{BS}(c,d)$ to $\ZZ$ is the group homomorphism that sends~$a$ to~$1$ and~$b$ to~$0$. Its kernel when restricted to $\mathrm{BS}(c,d)^+$ is nontrivial. It is the unital semigroup generated by~$b$. Nevertheless, the height map is controlled in the sense of \cite{aHRT}. Because the kernel of the height map gives an amenable quasi-lattice order, it follows from \cite{aHRT} that $\Cst(\mathrm{BS}(c,d),\mathrm{BS}(c,d)^+)$ is amenable.

Here we present a more general definition of controlled maps. The main advantage of this new definition is that it also allows us to treat examples of weak quasi-lattices which have an infinite descending chain $$p_1\ge p_2\ge \cdots\ge p_n\ge \cdots.$$ Our motivating example is the weak quasi-lattice order coming from the Baumslag--Solitar semigroup $$\mathrm{BS}(c,-d)^+=\langle a,b: b^dab^c =a\rangle^+$$ and the corresponding Baumslag--Solitar group $\mathrm{BS}(c,-d)$~\cite{SBS}, where $c,d\geq 1$. The height map from $\mathrm{BS}(c,-d)$ to $\ZZ$ is not controlled in the sense of an Huef, Raeburn and Tolich because it is constant on the infinite descending chain $$a>\cdots>b^{nd}a >b^{(n+1)d}a>\cdots.$$ This sequence is not bounded below in $\mathrm{BS}(c,-d)^+$ by an element of height~$1$. However, the height map in this example is controlled in our new sense (see Definition~\ref{defn-less-controlled}). Our main theorem is:

\begin{thm}\label{T:nuclear}
Let $(G,P)$ and $(K,Q)$ be weakly quasi-lattice ordered groups. Suppose that $\mu\colon(G,P)\to (K,Q)$ is a controlled map in the sense of Definition~\textup{\ref{defn-less-controlled}}, that $K$ is amenable and that $\Cst(\ker\mu,\ker\mu\cap P)$ is nuclear. Then $\Cst(G,P)$ is nuclear. In particular, $(G,P)$ is amenable.
\end{thm}

The proof involves a detailed analysis of the structure of a fixed-point algebra associated to a coaction of~$K$ on $\Cst(G,P)$, obtained from $\mu$. This fixed-point algebra is in general much larger than $\Cst(\ker\mu, \ker\mu\cap P)$, but we prove that it is also nuclear.  Then we apply \cite[Corollary~2.17]{Quigg-discrete}: if a $\Cst$\nb-algebra $B$ carries a coaction of a discrete amenable group for which the fixed-point algebra is nuclear, then~$B$ is nuclear, too.

Recently, Li, Omland and Spielberg studied semigroup $\Cst$\nb-algebras and their boundary quotients for right LCM one-relator semigroups. Their main result concerning nuclearity is  \cite[Corollary~3.10]{LOS}. Under certain assumptions on the defining relator, they provided a graph model for the boundary quotient semigroup $\Cst$\nb-algebra, and concluded that both the semigroup $\Cst$\nb-algebra and its boundary quotient are nuclear.  Motivated by~\cite{LOS}, we also look at one-relator semigroups. We obtain nuclearity of the semigroup $\Cst$\nb-algebras associated to a class of weakly quasi-lattice ordered groups coming from HNN extensions, where the controlled maps involved are also height maps. This contains all the Baumslag--Solitar semigroups as particular cases. Our nuclearity results for one-relator semigroups have no overlap with those in \cite{LOS}, and may be used in combination with \cite[Corollary~3.5]{LOS} and \cite{Tu1999} to produce examples of $\Cst$\nb-algebras that are classifiable by $K$\nb-theory.

Weak quasi-lattice orders are in general not well-behaved under semidirect products. In other words, the semidirect product of two weakly quasi-lattice ordered groups may not be weakly quasi-lattice ordered. We illustrate this with Example~\ref{ex:nonexample}. However, we show that particular examples of semidirect products are naturally weakly quasi-lattice ordered. These mostly involve actions on free groups $\FF$. We then use controlled maps and nuclearity of $\Cst(\FF,\FF^+)$ to obtain nuclearity for the corresponding semigroup $\Cst$\nb-algebras.

We also include graph products of weak quasi-lattices among our examples. We first show that the graph product of weakly quasi-lattice ordered groups is again weakly quasi-lattice ordered, with a natural positive cone. To do so, we prove a result along the same lines of \cite[Proposition~13]{CL1}. Our proof is a little different though: that the graph product of quasi-lattice orders is again quasi-lattice ordered was implicitly used in the proof of \cite[Proposition~13]{CL1} (see \cite[Theorem~10]{CL1}). When the underlying groups are amenable, we apply our result to obtain nuclearity of the semigroup $\Cst$\nb-algebra associated to the graph product of weak quasi-lattices.

After we submitted this article for publication,  we learned that Fountain and Kambites prove that a graph product of left LCM monoids is again a left LCM  monoid \cite[Theorem~2.6]{FK}.  Our Theorem~5.25 proves this result for the particular case of weak quasi-lattice orders,  and also gives an iterative procedure for computing the least upper bound for elements in a graph product. This procedure is crucial for finding a controlled map, and so we have retained Theorem~5.25 in order to prove our nuclearity result in Corollary~5.27. 
We thank Nadia Larsen for bringing \cite{FK} to our attention.

\subsection*{Outline} We start in \S\ref{sec-background} with some background material. In \S\ref{sec-control} we discuss the old and new definitions of controlled maps, and illustrate the need for greater generality with examples. We also answer in the negative a question, by Marcelo Laca, whether a positive cone of a weak quasi-lattice order can be embedded in group such that this new pair is quasi-lattice ordered. This question was motivated by Scarparo's example (Example~\ref{ex-scarparo}). In \S\ref{sec-nuclearity} we introduce the coaction built from a controlled map, analyse the structure of the fixed-point algebra, and then prove our main theorem (Theorem~\ref{thm-nuclearity2}). In \S\ref{sec:examples} we present our classes of examples. In \S\ref{sec-amenability}, we prove an amenability result which would be of interest when $(\ker\mu,\ker\mu\cap P)$ is amenable but $\Cst(\ker\mu,\ker\mu\cap P)$ is not known to be nuclear. We provide in the appendix a characterisation of nuclearity for $\Cst(G,P)$ through faithfulness of conditional expectations.

\section{Background}\label{sec-background}

Let $G$ be a discrete group with a  unital subsemigroup $P$ such that $P\cap P^{-1}=\{e\}$. There is a partial order on $G$ defined by $x\leq y$ if and only if $x^{-1}y\in P$.
It is easy to see that $\le$ is left-invariant.
In \cite{N}, Nica defined the pair $(G,P)$ to be \emph{quasi-lattice ordered} if the following is satisfied:
\begin{enumerate}
\item[(QL)]  any $n\geq 1$ and $x_1, x_2, \dots, x_n\in G$ which have common upper bounds in $P$ also have a least common upper bound in $P$.
\end{enumerate}
We denote the least common upper bound of $x_1,  \dots, x_n$ by  $x_1\vee\dots\vee x_n$ if it exists. In this case, we sometimes write $x_1\vee  \dots\vee x_n<\infty$. We then write $x_1\vee \dots\vee x_n=\infty$  if $x_1,  \ldots,x_n$ have no upper bound in $P$. 
As already observed by Nica, $(G,P)$ is quasi-lattice ordered if and only if
\begin{enumerate}
\item[(QL1)] any $x\in PP^{-1}$ has a least upper bound in $P$; and 
\item[(QL2)] any $x, y\in P$ with a common upper bound have a least common upper bound. 
\end{enumerate}
In fact, (QL) and (QL1) are equivalent.  See \cite[Lemma~7]{CL1} for this and further equivalent formulations of (QL). Only (QL2) is needed to define the $\Cst$\nb-algebras studied by Nica in \cite{N}.

In this paper, we will consider pairs $(G,P)$ as above satisfying only (QL2). We will use the terminology of Exel in \cite[Definition 32.1(v)]{Exel's book}:
\begin{dfn}
Let $G$ be a discrete group with a  unital subsemigroup $P$ such that $P\cap P^{-1}=\{e\}$ and with the induced left-invariant order on  $G$. If  $(G,P)$ satisfies (QL2), then we say that $(G,P)$ is a \emph{weakly quasi-lattice ordered group} or that $(G,P)$ is a \emph{weak quasi-lattice}.  
\end{dfn}

To make clear the relationship to right LCM semigroups, we note that we can rephrase (QL2) as follows:
\begin{enumerate}
\item[(QL2')]  for any $x, y\in P$ with $xP\cap yP\neq \emptyset$ there exists a unique $z\in P$ such that $xP\cap yP=zP$.
\end{enumerate}
In right LCM semigroups, the $z$ may not be unique. 

Consider $\ell^2(P)$ with orthonormal basis $\{e_x:x\in P\}$. For each $x\in P$, there is an isometry $T_x$ on $\ell^2(P)$ such that $T_xe_y=e_{xy}$ for $y\in P$. Because $P$ is left-cancellative,   $T: P\to B(\ell^2(P))$  is an isometric representation of $P$; it is called the \emph{Toeplitz representation} (and sometimes the left-regular representation).  Nica observed that
\begin{equation}\label{Nica}
T_xT_x^*T_yT_y^* =
\begin{cases}
T_{x \vee y}T_{x \vee y}^*&\text{if $x\vee y<\infty$}\\
0&\text{if $x\vee y=\infty$.}
\end{cases}
\end{equation}
 The \emph{Toeplitz algebra} is the $\Cst$\nb-algebra $\Tt(G,P)$ generated by the Toeplitz representation, that is, 
\[
\Tt(G,P)\coloneqq \Cst(\{T_x: x\in P\})\subseteq B(\ell^2(P)).
\]
An isometric representation of $P$ satisfying \eqref{Nica} is called \emph{Nica covariant}. 
There is a $\Cst$\nb-algebra, denoted $
\Cst(G,P)$, that is universal for Nica-covariant representations. It follows from Nica covariance that if $w$ is the universal Nica covariant representation in $\Cst(G,P)$, then  
\[
\Cst(G,P)=\clsp\{w_x w_y^*:x,y\in P\}.
\]
Notice that both $\Cst(G,P)$  and $\Tt(G,P)$ are defined using only (QL2). 

The Toeplitz representation  induces a surjection \[\pi_T:\Cst(G,P)\to \Tt(G,P).\]  Following Nica, we say that $(G,P)$ is \emph{amenable} if $\pi_T$
is an isomorphism.
The usual approach to prove the amenability of $(G,P)$ involves a conditional expectation. 

Let $B$ be a $\Cst$\nb-subalgebra of a $\Cst$\nb-algebra $C$. 
Then $E\colon C\to B$  is a \emph{conditional expectation} if $E$ is linear, positive and contractive such that  $E(b)=b$, $E(bc)=bE(c)$ and $E(cb)=E(c)b$ for all  $b\in B$ and $c\in C$.  Equivalently, $E\colon C\to B$ is a linear contractive map with~$E(b)=b$ for all~$b\in B$ by \cite{Tomiyama} (see also \cite[Theorem II.6.10.2]{Bla}). A conditional expectation is called \emph{faithful} if $E(c^* c)=0$ implies $c=0$.

Nica proves that there is a conditional expectation 
\begin{equation*}
E\colon \Cst(G,P)\to \clsp\{w_x w_x^*:x\in P\} \text{ such that } E(w_xw_y^*)=\begin{cases}w_xw_x^*&\text{if $x=y$}\\0&\text{else.}
\end{cases}
\end{equation*}
By \cite[Proposition on page~34]{N}, a quasi-lattice $(G, P)$ is amenable if and only if $E$ is faithful. Again,  Nica only needed  (QL2)  to define the amenability of $(G,P)$ and to characterise it with the conditional expectation $E$.
 Moreover, by Proposition~\ref{prop-iff-diagonal}, nuclearity of $\Cst(G,P)$ can also  be characterised in terms of faithfulness of the canonical conditional expectation \[E_{A,\max}\colon A\otimes_{\max} \Cst(G,P)\to A\otimes_{\max}  \clsp\{w_x w_x^*:x\in P\}\] for all unital $\Cst$\nb-algebras~$A$.

\section{Controlled maps -- then and now}\label{sec-control}

The first definition of a controlled map $\mu\colon(G,P)\to (K,Q)$ between quasi-lattice ordered groups appeared in  \cite[Proposition~4.2]{LR}; the existence of such a map into an amenable group $K$ implies that $(G,P)$ is amenable in the sense of Nica~\cite{N}. Because this first definition asked that $\ker\mu\cap P=\{e\}$, it does not apply to the quasi-lattices associated to the Baumslag-Solitar groups which were shown to be amenable in~\cite{CaHR}. The more general notion of controlled maps of   \cite[Definition~3.1]{aHRT} allows for a nontrivial kernel, and applies to the ``height map'' of most, but not all, quasi-lattice ordered groups associated to the  Baumslag-Solitar groups (see Examples~\ref{example-BS}  and~\ref{example-ilija} below). In Definition~\ref{defn-less-controlled} below we give the most general definition of controlled maps to date: in particular it allows the infinite descending chains appearing  in Example~\ref{example-ilija}. 

Here we start with Definition~\ref{def-controlled},  a slight modification of  \cite[Definition~3.1]{aHRT} which allows weak quasi-lattices (\cite[Definition~3.1]{aHRT} could have been stated in this generality).

\begin{dfn}\label{def-controlled}
Let $(G,P)$ and $(K,Q)$ be weak quasi-lattices.
Suppose that $\mu\colon G\to K$ is an order-preserving group homomorphism. 
Then $\mu$ is \emph{controlled}, in the sense of an Huef, Raeburn and Tolich,
 if it has the following properties:
    \begin{enumerate}
      \item\label{cm-1-new} For $x,y\in P$ with $x\vee y<\infty$, we have $\mu(x)\vee \mu(y)=\mu(x\vee y)$.
      \item\label{cm-2-new} Let $q\in Q$. Set
$\Sigma_q=\{\sigma\in \mu^{-1}(q)\cap P: \text{$\sigma$ is minimal}\}$.
\begin{enumerate}
\item\label{cm-3-new} If $x\in \mu^{-1}(q)\cap P$, then there exists $\sigma\in \Sigma_q$ such that $\sigma\leq x$.
\item\label{cm-4-new}   If  $\sigma,\tau\in\Sigma_q$, then $\sigma\vee\tau<\infty\Rightarrow\sigma=\tau$.
\end{enumerate}
\end{enumerate}
\end{dfn}

\begin{example}\label{ex-free-group}  Let $\FF$ be the free group on a  countable set of generators and let $\FF^+$ be the  unital free semigroup, that is $\FF^+$ is the smallest subsemigroup containing  the identity $e$ and the generators. Then $(\FF,\FF^+)$ is a quasi-lattice \cite[\S2.2]{N}. Let $\mu\colon\FF\to \ZZ$ be the homomorphism that sends each generator to $1$. To see that  $\mu\colon(\FF,\FF^+)\to (\ZZ,\NN)$ is a controlled map in the sense of Definition~\ref{def-controlled}, we note that if $x\in \FF^+$, then $\mu(x)$ is precisely the length of~$x$. Suppose that $x,y\in \FF^+$ with $x\vee y<\infty$. Then either $x=yp$ or $y=xq$ for some $p,q\in  \FF^+$. Say $x=yp$. Then $\mu(x)\vee\mu(y)=\mu(x)=\mu(x\vee y)$, giving (\ref{cm-1-new}). For (\ref{cm-2-new}) we observe that for each $n\in\NN$ we have $\Sigma_n=\mu^{-1}(n)\cap \FF^+$, and that no two distinct elements of $\Sigma_n$ are comparable. Thus $\mu$ is controlled in the sense of Definition~\ref{def-controlled}.  Here $\ker\mu\cap \FF^+=\{e\}$, and hence $\mu$ is a controlled map in the sense of Laca and Raeburn \cite[Proposition~4.2]{LR} as well.
\end{example}

Using Example~\ref{ex-free-group} it is easy to construct an example of a weak quasi-lattice with a controlled map. This example of a weak quasi-lattice which is not a quasi-lattice was obtained by Scarparo in \cite{Scarparo} and is discussed on page~249 of  \cite{Exel's book}. 

\begin{example} \label{ex-scarparo} Consider the free group $\FF$  on $2$ generators $a$ and $b$ and let $\PP=\{e, bw: w\in\FF^+\}$. Then $(\FF,\PP)$ is a weak quasi-lattice. Again, let $\mu\colon\FF\to \ZZ$ be the homomorphism that sends each generator to~$1$. Then $\mu$ is controlled  in the sense of Definition~\ref{def-controlled} with $\Sigma_n=\mu^{-1}(n)\cap \PP$.  Again $\ker\mu\cap \PP=\{e\}$.
\end{example}

Exel observed  on page~249 of  \cite{Exel's book} that $\PP=\{e, bw: w\in\FF^+\}$ is isomorphic to the positive cone  in the free group $\FF_\infty$ on infinitely many generators, and hence $(\FF_\infty, \PP)$ is a quasi-lattice ordered group. Prompted by a question by Marcelo Laca, we show in Proposition~\ref{prop-dilian} below that the positive cone $P$ of the Baumslag-Solitar group discussed in Example~\ref{ex-BS-with-chains}  cannot be embedded in a group $G$  such that $(G,P)$ is a quasi-lattice ordered group.

\begin{example}\label{example-BS} Let $c,d\geq 1$ and 
consider the Baumslag-Solitar group \[G=\langle a, b: ab^c=b^{d}a\rangle.\] Let $P$ be the  unital subsemigroup generated by $a$ and $b$. Then $(G,P)$ is quasi-lattice ordered by \cite[Theorem~2.11]{SBS}.
Let $\mu\colon G\to \ZZ$ be the ``height map'', that is, the homomorphism such that $\mu(a)=1$ and $\mu(b)=0$. 
We will show that $\mu$ is a controlled map in the sense of Definition~\ref{def-controlled}.  

To see that $\mu$ is order preserving, suppose that $x\leq y\in G$. Then $x^{-1}y\in P$ and 
$\mu(y)-\mu(x)=\mu(x^{-1}y)\in \mu(P)=\NN
$. Thus $\mu(x)\leq \mu(y)$ in $\ZZ$.

Let $x,y\in P$ with $x\vee y<\infty$. Say $\mu(y)\geq \mu(x)$. If $\mu(y)>\mu(x)$, then by \cite[Lemma~3.4]{CaHR} there exists $t\in\NN$ such that $x\vee y=yb^t$. So $\mu(x\vee y)=\mu(y)=\mu(x)\vee\mu(y)$.
If $\mu(y) =\mu(x)$, then by \cite[Lemma~3.4]{CaHR} there exists $t\in\NN$ such that either $x\vee y=yb^t$ or $x\vee y=xb^t$. This also implies $\mu(x\vee y)=\mu(x)\vee\mu(y)$.

For (\ref{cm-2-new}) we use the normal form for elements of $P$ to see that
\begin{align*}\mu^{-1}(q)\cap P&=\{b^{s_0}ab^{s_1}\cdots b^{s_{q-1}}ab^{s_{{q}}}: 0\leq s_0,\dots, s_{q-1}<d, s_{{q}}\in \NN\}\quad\text{and}\\
\Sigma_q&=\{ b^{s_0}ab^{s_1}\cdots b^{s_{q-1}}a : 0\leq s_0,\dots, s_{q-1}<d \}.
\end{align*}
Looking at the two sets, (\ref{cm-3-new}) clearly holds.  Let $x,y\in \Sigma_q$ with $x\vee y<\infty$. Then, by  \cite[Lemma~3.4]{CaHR},  either $x\vee y=yb^t$ or $x\vee y=xb^t$ for some $t$, and since $x,y\in \Sigma_q$  we get that  $x=y$.  This gives (\ref{cm-4-new}). 
\end{example}

The next example, provided by Ilija Tolich,   shows how  (\ref{cm-3-new}) in Definition~\ref{def-controlled} can fail.

\begin{example} \label{example-ilija}
Let $d\geq 1$ and consider the Baumslag-Solitar group \[G=\langle a, b: ab=b^{-d}a\rangle.\]
This fits into case (BS3) in \cite{SBS} with $c=1$.  Let $P$ be the subsemigroup generated by $a$ and $b$.  By  \cite[Lemma~2.12]{SBS}, $PP^{-1}=P\cup P^{-1}$, and it follows easily that $(G,P)$ is a quasi-lattice\footnote{At the end of the proof of \cite[Lemma~2.12]{SBS}, it is stated that $(G,P)$ is ``totally ordered, hence quasi-lattice ordered'', but $(G,P)$ is not totally ordered. For example, $b^ma\leq b^na$ if and only if $d$ divides $m-n$ and $n\leq m$. In particular, if $d$ does not divide $m-n$, then $b^na$ and $b^ma$ are not comparable.}.  

Let $\mu\colon G\to \ZZ$ be the height map as in the previous example. Then $\mu^{-1}(1)\cap P$ consists of elements of the form $b^iab^j$ for some $i,j\in\NN$, and using the relation $ab=b^{-d}a$ we can write that as $b^ma$ for some $m\in\ZZ$.
We have $b^ma\leq b^na$ if and only if $d$ divides $m-n$ and $n\leq m$. So for $0\leq k<d$ we  have the  infinite chains 
\begin{align*}
&\ldots\leq b^{3d+k}a\leq b^{2d+k}a\leq b^{d+k}a\leq b^ka\leq b^{-1d+k}a\leq b^{-2d+k}a\leq \ldots
\end{align*}
partioning $\mu^{-1}(1)\cap P$. Thus $\Sigma_1=\emptyset$, and then (\ref{cm-3-new}) in Definition~\ref{def-controlled} implies that $\mu$ is not a controlled map.  
\end{example}

\begin{dfn}\label{defn-less-controlled}
Let $(G,P)$ and $(K,Q)$ be weakly quasi-lattice ordered groups. Suppose that $\mu\colon G\to K$ is an order-preserving group homomorphism. We say that $\mu$ is \emph{controlled} if it has the following properties:
\begin{enumerate}
\item\label{defn-less-controlled-a}  For $x,y \in P$ with $x\vee y<\infty$, we have $\mu(x)\vee \mu(y)=\mu(x\vee y)$.
\item\label{defn-less-controlled-b}   For all $q\in Q$,  there exist a set $\Lambda_q$ and a family of sets $\{S_{\lambda}: \lambda \in \Lambda_q\}$ such that:
\begin{enumerate}
\item\label{defn-less-controlled-b1} $\mu^{-1}(q)\cap P$ is the disjoint union $\bigsqcup_{\lambda\in\Lambda_q}S_\lambda$; and 
\item\label{defn-less-controlled-b2} for each $\lambda\in \Lambda_q$, there is a decreasing sequence $(s_n^\lambda)_{n\in\NN}$ in $S_\lambda$  such that \[S_\lambda=\bigcup_{\substack{n\in\NN}}\{x\in  \mu^{-1}(q)\cap P \colon x\geq s_n^\lambda\}.\]
\end{enumerate}
\end{enumerate}
\end{dfn}

\begin{lem}\label{lem-vee-empty} Suppose that $\mu\colon G\to K$  is a controlled map in the sense of Definition~\textup{\ref{defn-less-controlled}}. Let $q\in Q$ and $\lambda_1\neq \lambda_2\in \Lambda_q$. Then $p_1\vee p_2=\infty$ for all $p_1\in S_{\lambda_1}$ and $p_2\in  S_{\lambda_2}$.
\end{lem}

\begin{proof}
Looking for a contradiction, we suppose that  there exist $x\in S_{\lambda_1}$ and $y\in S_{\lambda_2}$ such that $x\vee y<\infty$. Then $\mu(x\vee y)=\mu(x)\vee \mu(y)=q$ gives $x\vee y\in\mu^{-1}(q)\cap P$. Also, there exist $m,n\in \NN$ such that $x\geq s_m^{\lambda_1}$ and $y\geq s_n^{\lambda_2}$.
But now $x\vee y\geq x\geq s_m^{\lambda_1}$ implies that $x\vee y\in S_{\lambda_1}$. Similarly, $x\vee y\in S_{\lambda_2}$. This contradicts that $\mu^{-1}(q)\cap P$ is the disjoint union $\bigsqcup_{\lambda\in\Lambda_q}S_\lambda$.
\end{proof}

\begin{lem}
 Let $(G,P)$ and $(K,Q)$ be weakly quasi-lattice ordered groups.  Suppose that $\mu\colon G\to K$ is an order-preserving group homomorphism. If $\mu$ is a controlled map in the sense of Definition~\textup{\ref{def-controlled}}, then it is controlled in the sense of Definition~\textup{\ref{defn-less-controlled}}.
\end{lem}

\begin{proof}
Let $\mu$ be a controlled map in the sense of Definition~\ref{def-controlled}. We need to verify item (\ref{defn-less-controlled-b}) of Definition~\ref{defn-less-controlled}.
Fix $q\in Q$. Set $\Lambda_q=\Sigma_q.$
For each $\lambda\in\Lambda_q$, set
\[
S_\lambda=\{x\in \mu^{-1}(q)\cap P\colon x\geq \lambda\}.
\]
Suppose that $S_{\lambda_1}\cap S_{\lambda_2}\neq \emptyset$ and let $x\in S_{\lambda_1}\cap S_{\lambda_2}$. Then $x\geq \lambda_i$ and hence $\lambda_1\vee \lambda_2<\infty$. Since $\lambda_1, \lambda_2\in \Sigma_q$, this implies $\lambda_1=\lambda_2$ by item (\ref{cm-4-new}) of Definition~\ref{def-controlled}. Together with  (\ref{cm-3-new}) of Definition~\ref{def-controlled} we get that
\[\mu^{-1}(q)\cap P=\bigsqcup_{\lambda\in\Lambda_q} S_\lambda\]
which is item (\ref{defn-less-controlled-b1}) of Definition~\ref{defn-less-controlled}.
For item (\ref{defn-less-controlled-b2}) of Definition~\ref{defn-less-controlled} we just take the sequence $(s_n^\lambda)_{n\in\NN}$ to be the constant sequence $(\lambda)_{n\in\NN}$.
\end{proof}

\begin{example}\label{ex-BS-with-chains}
Let $c,d\geq 1$ and consider the Baumslag--Solitar group 
\[
G=\langle a,b\colon ab^c=b^{-d}a\rangle.
\] 
Let $P$ be the (unital) subsemigroup of~$G$ generated by~$a$ and~$b$. Then $(G,P)$ is a weakly quasi-lattice ordered group
by  \cite[Proposition~2.10]{SBS}  and it is a quasi-lattice ordered group if and only if~$c=1$  by \cite[Lemma~2.12]{SBS}.  

Let $\mu\colon G\to \ZZ$ be the height map as in Examples~\ref{example-BS} and \ref{example-ilija}. Since $\mu(P)\subseteq \NN$, it follows that $\mu$ is order-preserving. We will show that $\mu$ is a controlled map from $(G,P)$ to $(\ZZ,\NN)$ in the sense of Definition~\ref{defn-less-controlled}.

To do this, we will use  the fact that every $x\in P$ with height $\mu(x)=q$ has a unique normal form
\begin{equation}\label{eq:formL}
x=b^{s_0}ab^{s_1}\cdots b^{s_{q-1}}ab^{s_q}
\end{equation}
where $0\leq s_0,\dots, s_{q-1}<d$, $s_q\in\ZZ$ and $\mu(x)=0\Longrightarrow s_q\geq 0$ (see \cite[Proposition~2.1 and Remark~2.4]{SBS}). 

 Let $x,y\in P$ with~$x\vee y<\infty$. We claim that $x$ and $y$ are comparable.  To see this, we consider $x, y$ in normal form \eqref{eq:formL}, say
\[
x=b^{s_0}ab^{s_1}\cdots b^{s_{q-1}}ab^{s_q} \text{\ and\ } y=b^{t_0}ab^{t_1}\cdots b^{t_{r-1}}ab^{t_r}
\]
Without loss of generality, suppose that $\mu(x)\leq \mu(y)$ (otherwise switch them). Then $s_i=t_i$ by \cite[Proposition~2.10]{SBS} for $i\leq q-1$, and hence
\[
x^{-1}y=b^{t_q-s_q}ab^{t_{q+1}}a\cdots b^{t_{r-1}}ab^{t_r}.
\]
If $t_q-s_q\geq 0$, then clearly $x^{-1}y\in P$ and $x\leq y$. Suppose that $t_q-s_q< 0$. Choose $n\in\NN$ such that $dn\geq s_q-t_q$.
Then using the relation $a=b^dab^c$ repeatedly we get that
\[
x^{-1}y=b^{nd+t_q-s_q}ab^{nc}b^{t_{q+1}}a\cdots b^{t_{r-1}}ab^{t_r}.
\]
Since $b^{nd+t_q-s_q}ab^{nc}\in P$ and $b^{t_{q+1}}a\cdots b^{t_{r-1}}ab^{t_r}\in P$, so is $x^{-1}y$, and $x\leq y$ as claimed. 
Now $\mu(x\vee y)=\mu(y)=\mu(x)\vee \mu(y)$ and~$\mu$ satisfies (\ref{defn-less-controlled-a}) of Definition~\ref{defn-less-controlled}.

\medskip

For (\ref{defn-less-controlled-b}), we start by observing that $\mu^{-1}(0)\cap P=\{b^{n} \colon n\in\NN\}$. Then set $\Lambda_0\coloneqq\{e\}$, $S_e\coloneqq\mu^{-1}(0)\cap P$ and $s_n^e:=e$ for all~$n$.

Let $q\in\NN\setminus\{0\}$.  Using the normal form and the observation in \cite[Lemma~2.6]{SBS} that  $x\in P$ and $\mu(x)>0$ implies $xb^m\in P$ for all $m\in\ZZ$, we deduce that
\[\mu^{-1}(q)\cap P=\{b^{s_0}ab^{s_1}\cdots b^{s_{q-1}}ab^{m}\colon 0\leq s_0,\dots, s_{q-1}<d, m\in\ZZ\}.
\]
Set  $\Lambda_q\coloneqq [0,d)^q\cap\NN^q$. For each $\lambda=(s_0,s_1,\ldots,s_{q-1})\in\Lambda_q$, set 
\[
S_\lambda\coloneqq\{b^{s_0}ab^{s_1}\cdots b^{s_{q-1}}ab^{m}\colon m\in\ZZ\}.
\]
Then $S_{\lambda} \cap S_{\lambda'}=\emptyset$ if $\lambda\neq\lambda'$ by uniqueness of the normal form~\eqref{eq:formL}. This gives (\ref{defn-less-controlled-b1}). 

For (\ref{defn-less-controlled-b2}), let  $\lambda=(s_0,s_1,\ldots,s_{q-1})\in\Lambda_q$ and $n\in \NN$, and set 
\[
s_n^\lambda\coloneqq b^{s_0}ab^{s_1}\cdots b^{s_{q-1}}ab^{-n}.
\] 
If $m\leq n$, then  $(s_n^\lambda)^{-1}s_m^\lambda= b^{n-m}\in P$, and   $(s_n^\lambda)_{n\in\NN}$ is a decreasing sequence in $S_\lambda$. 
It remains to verify that
 \begin{equation}\label{eq-set-equality}
 S_\lambda=\bigcup_{\substack{n\in\NN}}\{x\in \mu^{-1}(q)\cap P\colon x\geq s_n^\lambda\}.
 \end{equation}
 Consider  $x=b^{s_0}ab^{s_1}\cdots b^{s_{q-1}}ab^{m}\in S_\lambda$ where $m\in\ZZ$. Choose $n\in\NN$ such that $n+m\geq 0$. Then $(s_n^\lambda)^{-1}x=b^{n+m}\in P$. On the other hand, let $x\in \mu^{-1}(q)\cap P$ and suppose that $x\geq s_n^\lambda$ for some $n\in\NN$. Then $(s_n^\lambda)^{-1}x\in P$. Since $\mu((s_n^\lambda)^{-1}x)=0$ we get $(s_n^\lambda)^{-1}x=b^m$ for some $m\in\NN$. Thus $x=s_n^\lambda b^m =b^{s_0}ab^{s_1}\cdots b^{s_{q-1}}ab^{m+n}\in S_\lambda$.
 This proves \eqref{eq-set-equality}. Thus $\mu$ is controlled in the sense of Definition~\ref{defn-less-controlled}. 
\end{example}

Prompted by the weak quasi-lattice of Example~\ref{ex-scarparo}, Marcelo Laca asked the following question: let $(G,P)$ be a weakly quasi-lattice ordered group. Can $P$ be always embedded in a group $H$ such that $(H,P)$ is quasi-lattice ordered?  Here we show by example that the answer to this question is ``no''.

\begin{prop}\label{prop-dilian}
Let  $c>1$ and $d\geq 1$ such that  $c$  does not divide $d$. Consider the Baumslag--Solitar group 
\[
G=\langle a,b\colon ab^c=b^{-d}a\rangle.
\] 
Let $P$ be the unital subsemigroup of~$G$ generated by~$a$ and~$b$. Then there is no embedding of $P$ in a group $H$ such that $(H,P)$ is quasi-lattice ordered. 
\end{prop}

\begin{proof} Suppose that there were a group $H$ and an embedding of $P$ in $H$. 
We identify $P$ with its copy in  $H$. To show that $(H,P)$ is not quasi-lattice ordered, we will show that 
 \[h=a(ab^{d})^{-1}\in PP^{-1}\]  does not have a least upper bound in $P$.

For all $n\in\NN$ we have
\[h^{-1}b^{nd}a=ab^{d}(a^{-1}b^{nd}a)=ab^{d}b^{-nc}=ab^{-nc}b^{d}=b^{nd}ab^{d}\in P.
\]
Thus  $h\leq b^{nd}a$ for all  $n\in\NN$, and  we have an infinite descending chain  \[(h\leq) \cdots \le b^{(n+1)d}a\leq b^{nd}a\le \cdots.\]
We will show that no element $x\in P$ satisfies 
\[
h\le x\le b^{nd}a\text{ for all }n\in\NN.
\] 
Assume, looking for a contradiction, that there exists $x\in P$ such that  $h \le x\le b^{nd}a$ for all  $n\in\NN$. 
Since $e\le x\le b^{nd}a$ and the height map $\mu\colon G\to\ZZ$ preserves the order, we have either $\mu(x)=0$ or $1$.  We will consider these two cases separately. 

First, suppose that $\mu(x)=0$. Then $x=b^k$ for some $k\in\NN$ by the normal form for elements of $P$. Since $h\le x$, we have $x=h p$ for some $p\in P$. In particular, $\mu(p)\ge 0$. We claim that  $\mu(p)=0$. Assume, looking for a contradiction,  that $\mu(p)>0$. Then \[h^{-1}=p x^{-1}=p b^{-k},\]   and the normal form for elements of  $G$ with positive height implies that $h^{-1}\in P$. 
Now
\[
h^{-1}=pb^{-k}\Longleftrightarrow ab^d=pb^{-k}a,
\]
and it makes sense to apply $\mu$ to the elements $ab^d$, $pb^{-k}a$, $pb^{-k}$ and $a$ of $P$.
But then
$1=\mu( ab^d)=\mu(pb^{-k}a) =\mu(pb^{-k})+\mu(a)>1$, a contradiction. It follows that $\mu(p)=0$, as claimed.
 Again, using the normal form, we have  $p=b^l$ for some $l\in\NN$.  Now 
$ab^d=b^lb^{-k}a=b^{l-k}a$, which implies that $ab^{-d}=b^{k-l}a$.  But this is impossible since $c\nmid d$.

Second, suppose that $\mu(x)=1$.  Since $x\le b^{nd}a$ for each $n\in\NN$, there exists $p_n\in P$ such that  we have $b^{nd}a= x p_n$. Then  $\mu(p_n)=0$. 
By the normal form, one has $x=b^{s_0}ab^{s_1}$ with $0\le s_0<d$ and $s_1\in \ZZ$ and $p_n=b^{\ell_n}$ for some $\ell_n\in \NN$. Now
\[
b^{nd}a= x p_n\Longleftrightarrow b^{nd} a = b^{s_0}ab^{s_1} b^{\ell_n}\Longleftrightarrow ab^{-nc}=b^{s_0}ab^{s_1+l_n}.
\]
This implies that $s_0=0$ and $-nc=s_1+l_n$ for all~$n\in \NN$ by uniqueness of normal forms. But now $\ell_n=-n c-s_1\to-\infty$ as $n\to\infty$, contradicting that $\ell_n\in\NN$. 

We conclude that no $x\in P$ satisfies 
$
h\le x\le b^{-nd}a
$
for all $n\in\NN$. In particular, $h$ has no least upper bound in~$P$. Thus $(H, P)$ is not quasi-lattice ordered. 
\end{proof}

\begin{lem}
Let $(G,P)$ and $(K,Q)$ be weakly quasi-lattice ordered groups. Suppose that $\mu\colon G\to K$ is an order-preserving group homomorphism. 
\begin{enumerate}
\item\label{lem-kernel-a} If $(G,P)$ is a quasi-lattice ordered group, then so is $(\ker\mu, \ker\mu\cap P)$.
\item Suppose that for $x,y\in P$ with $x\vee y<\infty$, we have $\mu(x)\vee \mu(y)=\mu(x\vee y)$. Then $(\ker\mu, \ker\mu\cap P)$ is a weakly quasi-lattice ordered group.
\end{enumerate}
\end{lem}

\begin{proof}  This is what was actually proved in \cite[Lemma~3.3]{aHRT} though its statement consists only of item (\ref{lem-kernel-a}).
\end{proof}

\section{Nuclearity} \label{sec-nuclearity}

Let $(G,P)$ be a weakly quasi-lattice ordered group. Suppose that there are a discrete group $K$ and a group homomorphism $\mu\colon G\to K$. We write $u$ for the universal unitary representation generating $\Cst(K)$.  By Lemma~3.4 of \cite{aHRT} there is a nondegenerate\footnote{Full coactions of discrete groups may not be nondegenerate -- see the Erratum \cite{KQ-erratum}.  While nondegeneracy of $\delta_\mu$ is not mentioned in the statement of \cite[Lemma~3.4]{aHRT}, it is in fact proved there.} injective coaction $\delta_\mu\colon \Cst(G,P)\to \Cst(G,P)\otimes_\mint \Cst(K)$ such that
\begin{equation}
\label{coaction-from-mu}
\delta_\mu(w_p)=w_p\otimes u_{\mu(p)}\quad\text{for $p\in P$}.
\end{equation}

Let $\tau$ be the trace on $\Cst(K)$ such that $\tau(u_k)$ is $1$ if $k=e$ and $0$ otherwise. 
By \cite[Lemma~3.5]{aHRT}, 
\begin{equation}
\label{expectation-from-mu}
\Psi_\mu=(\id\otimes\tau)\circ\delta_\mu
\end{equation}
is a conditional expectation on $\Cst(G,P)$ such that 
\[\Psi_\mu(w_pw_q^*)=
\begin{cases}w_pw_q^* &\text{if $\mu(p)=\mu(q)$}\\0&\text{else},
\end{cases}
\]
and the range of $\Psi_\mu$  is \[\clsp\{w_pw_q^*:p,q\in P,\mu(p)=\mu(q)\}.\] Moreover, if $K$ is amenable, then $\Psi_\mu$ is faithful.

\begin{prop}\label{prop-iff-fpa-mu-cutdown} Let $(G,P)$  be a weak quasi-lattice. Suppose that there exist an amenable group $K$ and a homomorphism $\mu\colon G\to K$. Then $\Cst(G,P)$  is nuclear
if and only if $\clsp\{w_pw_q^*: p,q\in P\text{\ and\ }\mu(p)=\mu(q)\}$ is nuclear.
\end{prop}

\begin{proof} Let $\delta_\mu$ be the coaction defined by \eqref{coaction-from-mu}.  The proposition follows from \cite[Corollory~2.17]{Quigg-discrete} after observing that  $\clsp\{w_pw_q^*:p,q\in P,\mu(p)=\mu(q)\}$ is the fixed-point algebra $\Cst(G,P)^{\delta_\mu}$ of the cosystem $(\Cst(G,P), K,\delta_\mu)$. (The amenability of $K$ is only needed for the direction $\Cst(G,P)^{\delta_\mu}$ nuclear implies $\Cst(G,P)$ nuclear.) 

For the observation, we recall from \cite[Lemma~1.3]{Quigg-discrete} that the fixed-point algebra $\Cst(G,P)^{\delta_\mu}$ of  $(\Cst(G,P), K,\delta_\mu)$ is 
\[
\Cst(G,P)^{\delta_\mu}=\{a\in \Cst(G,P):\delta_\mu(a)=a\otimes u_e\}.
\]
Let $p,q\in P$ such that $\mu(p)=\mu(q)$. Then $\delta_\mu(w_pw_q^*)=w_pw_q^*\otimes u_{\mu(p)} u_{\mu(q)}^*=w_pw_q^*\otimes u_e$, and it follows that $\clsp\{w_pw_q^*:p,q\in P,\mu(p)=\mu(q)\}\subseteq \Cst(G,P)^{\delta_\mu}$. 

For the other inclusion, let $a\in \Cst(G,P)^{\delta_\mu}$. Let $\Psi_\mu$ be the conditional expectation at \eqref{expectation-from-mu}. Then
\[\Psi_\mu(a)=(\id\otimes\tau)\circ \delta_\mu(a)=\id\otimes\tau(a\otimes u_e)=a,\]
and so $a$ is in the range $\clsp\{w_pw_q^*:p,q\in P,\mu(p)=\mu(q)\}$ of $\Psi_\mu$. Thus \[\clsp\{w_pw_q^*:p,q\in P,\mu(p)=\mu(q)\}=\Cst(G,P)^{\delta_\mu}\] as needed.
\end{proof}

\begin{lem}\label{lem-subalgebra}  Let $(G,P)$  be a weak quasi-lattice. Suppose that there exist a group $K$ and a homomorphism $\mu\colon G\to K$ such that $(\ker\mu, \ker\mu\cap P)$ is weakly quasi-lattice ordered. 
\begin{enumerate}
\item\label{lem-subalgebra-a}   If   $(\ker\mu, \ker\mu\cap P)$  is amenable, then $\Cst(\ker\mu, \ker\mu \cap P)$ is isomorphic to the $\Cst$\nb-subalgebra \[
B_e\coloneqq\clsp\{w_\alpha w_\beta^*:\alpha,\beta\in P, \mu(\alpha)=e=\mu(\beta)\}
\]
of $\Cst(G,P)$; the isomorphism is induced by the restriction of the universal representation $w:P\to \Cst(G,P)$ to $\ker\mu\cap P$. 
\item\label{lem-subalgebra-b}  If $B_e$ is nuclear, then $\Cst(\ker\mu, \ker\mu \cap P)$ is nuclear. In particular, $B_e$ and $\Cst(\ker\mu, \ker\mu \cap P)$ are isomorphic and $(\ker\mu, \ker\mu\cap P)$  is amenable. 
\end{enumerate}
\end{lem}

\begin{proof} 
The restriction $w|$ of $w:P\to \Cst(G,P)$ to $\ker\mu\cap P$ is a Nica covariant representation of $\ker\mu\cap P$, and hence induces a representation $\pi_{w|}: \Cst(\ker\mu, \ker\mu \cap P)\to B_e$ which is surjective.  Let $T$ and $S$  be the Toeplitz representations of $P$ and $\ker\mu\cap P$ on $\ell^2(P)$ and $\ell^2(\ker\mu\cap P)$, respectively. 
 We view $\ell^2(\ker\mu\cap P)$ as a closed subspace of $\ell^2(P)$. Since $\ell^2(\ker\mu\cap P)$ is invariant under $B_e$, we obtain a representation $\pi^\mu_T$ of $B_e$ on $\ell^2(\ker\mu\cap P)$ by restriction. Then $\pi^\mu_T \circ \pi_{w|}=  \pi_S$.

For (\ref{lem-subalgebra-a}), suppose that  $(\ker\mu, \ker\mu\cap P)$ is amenable. Then $\pi_S$ is injective and then  $\pi_{w|}$ is the required isomorphism.

For  (\ref{lem-subalgebra-b}), suppose that $B_e$ is nuclear.  Since $\pi^\mu_T$ is surjective, $\Tt(\ker\mu, \ker\mu \cap P)$ is a quotient of $B_e$ and hence is nuclear as well.  We claim that $\ker\mu\cap P$ satisfies the independence condition of \cite[Definition~6.30]{Li-book}. It then follows from \cite[Theorem~6.44]{Li-book} that  $\Cst(\ker\mu, \ker\mu \cap P)$ is nuclear and that $(\ker\mu, \ker\mu \cap P)$ is amenable. Then item~(\ref{lem-subalgebra-a}) completes the proof of ~(\ref{lem-subalgebra-b}). 

To see that $\ker\mu\cap P$ satisfies the independence condition and that \cite[Theorem~6.44]{Li-book} applies we need to do some reconciling.  In \cite{Li-book}, the full $\Cst$\nb-algebra $\Cst(Q)$ of a left-cancellative semigroup $Q$ is defined as a full inverse semigroup $\Cst$\nb-algebra (see also~\cite{Norling}). It is pointed out at the end of \S6.6 of \cite{Li-book} that this is the same as the $\Cst$\nb-algebra $\Cst_S(Q)$ introduced in  \cite{Li-JFA}.   This $\Cst_S(Q)$ is a quotient of a $\Cst$\nb-algebra  $\Cst(Q)$ defined in \cite[Definition~3.2]{Li-JFA}.  If $(K,Q)$ is a weak quasi-lattice, then the latter $\Cst(Q)$ is the $\Cst$\nb-algebra $\Cst(K,Q)$ universal for Nica-covariant representations  that we are studying in this paper (see \cite[p.~4313]{Li-JFA}). Further, on page~4327 of \cite{Li-JFA}, Li shows that if $(K,Q)$ is a weak quasi-lattice, then $\Cst(Q)$ and $C_S^*(Q)$ are isomorphic.  If $(K,Q)$ is a weak quasi-lattice, then $Q$ satisfies the independence condition (see Definition~6.20 and Lemmas~6.31 and~6.32 of \cite{Li-book}).  Thus  \cite[Theorem~6.4]{Li-book} applies as claimed. 
\end{proof}

It follows from  Lemma~\ref{lem-subalgebra}(\ref{lem-subalgebra-b}) that we can restate Theorem~\ref{T:nuclear} as:

\begin{thm}\label{thm-nuclearity2}
Let $(G,P)$ and $(K,Q)$ be weakly quasi-lattice ordered groups.
Suppose that $\mu\colon(G,P)\to (K,Q)$ is a controlled map in the sense of Definition~\textup{\ref{defn-less-controlled}}, that $K$ is amenable and that  the $\Cst$\nb-subalgebra
\[
B_e=\clsp\{w_\alpha w_\beta^*:\alpha,\beta\in P, \mu(\alpha)=e=\mu(\beta)\}
\]
of $\Cst(G,P)$ is nuclear. Then $\Cst(G,P)$ is nuclear. In particular, $(G,P)$ is amenable.
\end{thm}

Theorem~\ref{thm-nuclearity2} extends \cite[Corollary~8.3]{Li} in two ways: our notion of controlled maps is weaker and we only assume that $(G,P)$ and $(K,Q)$ are weak quasi-lattices. Before we prove the theorem, we observe that the hypothesis that the $\Cst$\nb-subalgebra $B_e$ is nuclear is often easily seen to hold. It is equivalent to nuclearity of $\Cst(\ker\mu, \ker\mu \cap P)$ by Lemma~\ref{lem-subalgebra} since this has $B_e$ as a quotient.

\begin{cor}\label{cor-kernel} Let $(G,P)$ and $(K,Q)$ be weakly quasi-lattice ordered groups.
Suppose that $\mu\colon(G,P)\to (K,Q)$ is a controlled map in the sense of Definition~\textup{\ref{defn-less-controlled}}. 
If the group generated by $\ker\mu\cap P$ is amenable, then $\Cst(\ker\mu, \ker\mu \cap P)$ is nuclear.
\end{cor} 

\begin{proof} 
Let $G'$ be the subgroup of $G$ generated by $P'\coloneqq\ker\mu\cap P$. Then the identity map on $(G', P')$ is a controlled map.  Here $\CC=\Cst(\ker\id\cap P')\cong B_e(G', P')$ is nuclear. If  $G'$ is amenable, Theorem~\ref{thm-nuclearity2}  implies that $\Cst(\ker\mu, \ker\mu \cap P)$ is nuclear. 
\end{proof}

In view of the new definition of controlled maps, to prove Theorem~\ref{thm-nuclearity2} we  need to re-analyse the structure of the fixed-point algebra 
\[\Cst(G,P)^{\delta_\mu}=\clsp\{w_pw_q^*: p,q\in P\text{\ and\ }\mu(p)=\mu(q)\}.\]  For this, we start with the analogues of \cite[Lemmas~3.6 and~3.9]{aHRT}.

 \begin{prop}\label{prop-structure-fpa}
 Let $(G,P)$ and $(K,Q)$ be weakly quasi-lattice ordered groups.
Suppose that $\mu\colon(G,P)\to (K,Q)$ is a controlled map in the sense of Definition~\textup{\ref{defn-less-controlled}}. 
\begin{enumerate} 
\item
 Let $k\in Q$, let $F$ be a finite subset of $\Lambda_k$ and let $n\in\NN$. Set
\begin{gather*}
B_{k}\coloneqq\clsp\{w_p w_q^*: p,q\in P, \mu(p)=k=\mu(q)\},\\
B_{k,n}\coloneqq\clsp\{w_{s^\lambda_n} dw^*_{s^{\rho}_n}\colon \lambda,\rho\in\Lambda_k \text{\ and\ }d\in B_e\},\\
B_{k,n, F}\coloneqq\lsp\{w_{s^\lambda_n}dw^*_{s^{\rho}_n}\colon \lambda,\rho\in F \text{\ and\ }d\in B_e\}.
\end{gather*}
Then 
\begin{enumerate} \item\label{prop-structure-fpa-ai} $B_k$, $B_{k,n}$ and $B_{k,n, F}$ are $\Cst$\nb-subalgebras of $\Cst(G,P)^{\delta_\mu}$; 
\item\label{prop-structure-fpa-aii} $B_{k,n, F}$ is isomorphic to $M_{F}(\CC)\otimes B_e$;
\item\label{prop-structure-fpa-aiii} $B_{k,n}=\varinjlim B_{k,n,F}$  is isomorphic to $\Kk(\ell^2(\Lambda_k))\otimes B_e$;
\item\label{prop-structure-fpa-aiv} 
$B_{k,n}\subseteq B_{k,n+1}$ and $B_k=\varinjlim B_{k,n}$. 
\end{enumerate}
\item
 Let $\Ii$ be the set of all finite sets $I\subseteq Q$ that are closed under $\vee$ in the sense that $s,t\in I$ and $s\vee t<\infty$ implies that $s\vee t\in I$.
Let $I\in\Ii$ and set \[C_I\coloneqq \overline\lsp\{w_pw^*_q:\mu(p)=\mu(q)\in I\}.\] 
 Then 
 \begin{enumerate}
 \item\label{prop-structure-fpa-bi} $C_I$ is a $\Cst$\nb-subalgebra of $\Cst(G,P)^{\delta_\mu}$;
 \item\label{prop-structure-fpa-bii}  $C_I=\sum_{k\in I} B_k$\footnote{That $C_I=\sum_{k\in I} B_k$  was also asserted in  \cite[Lemma~3.9]{aHRT}, but the proof given there is not correct; the mistake is in the  statement that ``the finite span of closed subalgebras is closed''.};   and 
 \item\label{prop-structure-fpa-biii} $\Cst(G,P)^{\delta_\mu}=\varinjlim_{I\in\Ii}C_I$.
 \end{enumerate}
\end{enumerate}
\end{prop}

\begin{proof} To see that  $B_k$ is a $\Cst$\nb-subalgebra, take spanning elements $w_pw^*_q$ and $w_rw_s^*\in B_k$. Then, using Nica covariance, we deduce that 
\[
w_pw_q^*w_rw_s^*=\begin{cases}
w_{pq^{-1}(q\vee r)}w_{sr^{-1}(q\vee r)}^*&\text{if $q\vee r<\infty$}
\\
0&\text{else.}
\end{cases}
\] 
Suppose that $q\vee r<\infty$.  Then $w_pw_q^*w_rw_s^*\neq 0$.  Since $\mu$ is a $\vee$-preserving homomorphism we have
\[
\mu(pq^{-1}(q\vee r))=\mu(p)\mu(q)^{-1}\big(\mu(q)\vee\mu(r)\big)=kk^{-1}k=k
\] 
and $\mu(sr^{-1}(q\vee r))=k$ similarly. Thus $w_pw_q^*w_rw_s^*\in B_k$ and it follows that $B_k$ is a $\Cst$\nb-subalgebra of  $\Cst(G,P)^{\delta_\mu}$. Similarly, using Nica covariance and Lemma~\ref{lem-vee-empty}, we have 
\begin{align*}
(w_{s^\lambda_n} w_\alpha &w_\beta^*w^*_{s^{\rho}_n})(w_{s^{\lambda'}_n} w_\sigma w_\tau^*w^*_{s^{\rho'}_n})\\
&=\begin{cases}
w_{s^\lambda_n} w_{\alpha \beta^{-1}(\beta\vee\sigma)}w^*_{\tau\sigma^{-1}(\beta\vee \sigma)}w^*_{s^{\rho'}_n}
&\text{ if $\rho=\lambda'$ and $\beta\vee\sigma<\infty$}\\
0&\text{else.}
\end{cases}
\end{align*}
Since $\mu(\alpha \beta^{-1}(\beta\vee\sigma))=e=\mu(\tau\sigma^{-1}(\beta\vee \sigma))$ it follows that $B_{k,n}$ is a $\Cst$\nb-subalgebra of  $\Cst(G,P)^{\delta_\mu}$.

Next we show that $B_{k,n, F}$ is isomorphic to  $M_{F}(\CC)\otimes B_e$. This will complete the proof of (\ref{prop-structure-fpa-ai}) as well as giving (\ref{prop-structure-fpa-aii}). 
Let $\lambda, \lambda', \rho,\rho'\in F$. Then
\[
w_{s_n^\lambda}w^*_{s_n^{\rho}}w_{s_n^{\lambda'}}w^*_{s_n^{\rho'}}=\begin{cases}
w_{s_n^\lambda}w^*_{s_n^{\rho'}} &\text{if $\rho=\lambda'$}\\
0&\text{else}
\end{cases}
\]
using Lemma~\ref{lem-vee-empty}.
Thus $\{w_{s_n^\lambda}w^*_{s_n^{\rho}}\colon \lambda,  \rho\in F\}$ is a set of matrix units in $B_{k,n}$. 
This gives an injective homomorphism $\theta\colon M_{F}(\CC)\to \overline{B_{k,n,F}}$ that maps the matrix units $\{E_{\lambda,\rho}\colon \lambda,\rho\in F\}$  in $M_{F}(\CC)$ to $\{w_{s_n^\lambda}w^*_{s_n^{\rho}}\colon \lambda,\rho\in F\}$ in $B_{k,n,F}$. It is  easy to check that the formula   
\[
\psi(d)=\sum_{\lambda\in F}w_{s_n^\lambda}dw^*_{s_n^\lambda}
\] 
gives an injective  homomorphism $\psi:B_e\to \overline{B_{k,n,F}}$ such that 
\[\theta(E_{\lambda,\rho})\psi(d)=w_{s_n^\lambda}d w_{s_n^{\rho}}^*=\psi(d)\theta(E_{\lambda,\rho}).\]
It follows that the ranges of $\theta$ and $\psi$ commute. Now the universal property of the maximal tensor product gives a homomorphism $\theta\otimes \psi$ of $M_{F}(\CC)\otimes B_e$ into $\overline{B_{k,n,F}}$.

By \cite[Theorem B.18]{tfb} 
\[
M_{F}(\CC)\otimes B_e
=\lsp\{E_{\lambda,\rho}\otimes d: \lambda,\rho\in F \text{\ and\ } d\in B_e\},
\]
 with no closure. So the range of $\theta\otimes \psi$ is spanned by $\theta(E_{\lambda,\rho})\psi(d)=w_{s_n^{\lambda}}dw^*_{s_n^{\rho}}$, and hence the range  is  $B_{k,n,F}$. (In particular,  $B_{k,n,F}$ is  closed.)  Thus  $B_{k,n, F}$  is isomorphic to $M_{F}(\CC)\otimes B_e$.
 
 \medskip

Now direct the finite subsets of $\Lambda_k$ by inclusion. Consider a spanning element $w_{s^\lambda_n} d w^*_{s^\rho_n}\in B_{k,n}$. Take $F=\{\lambda,\rho\}$. Then $w_{s^\lambda_n} d w^*_{s^{\rho}_n}\in B_{k,n, F}$. Thus $B_{k,n}=\varinjlim_F B_{k,n, F}$. It follows that $B_{k,n}$  is isomorphic to $\Kk(\ell^2(\Lambda_k))\otimes B_e$, giving  (\ref{prop-structure-fpa-aiii}).

\medskip
To see that $B_{k,n}\subseteq B_{k,n+1}$, consider a  spanning element $w_{s^\lambda_n} w_\alpha w_\beta^*w^*_{s^\rho_n}\in B_{k,n}$. We have $s_{n+1}^\lambda\leq s_n^\lambda$, that is, there exists $\sigma\in P$ such that $s_n^\lambda=s_{n+1}^\lambda \sigma$. Notice that $\mu(\sigma)=e$ because $\mu(s_n^\lambda)=k=\mu(s_{n+1}^\lambda)$. Similarly, there exists $\tau\in\ker\mu\cap P$ such that $s_n^\rho=s_{n+1}^\rho \tau$. Now 
$w_{s^\lambda_n} w_\alpha w_\beta^*w^*_{s^\rho_n}=w_{s^\lambda_{n+1}} w_{\sigma\alpha} w_{\tau\beta}^*w^*_{s^\rho_{n+1}}\in B_{k, n+1}$. Thus $B_{k,n}\subseteq B_{k,n+1}$ 

Now take a spanning element $w_pw_q^*\in B_k$. By item (\ref{defn-less-controlled-b2}) of Definition~\ref{defn-less-controlled} there exist $\lambda, \rho\in \Lambda_k$ and $m,n\in\NN$ such that $s_m^\lambda\leq p$ and $s^\rho_n\leq q$. Say $m\leq n$. Then  $s_n^\lambda\leq  s_m^\lambda\leq p$. So there exist $\alpha,\beta\in\ker\mu\cap P$ such that $p=s_n^\lambda\alpha$ and $q=s_n^\rho\beta$.
Now $w_pw_q^*=w_{s_n^\lambda}w_\alpha w_\beta^* w_{s_n^\rho}^*\in B_{k,n}$. Thus $B_k=\varinjlim B_{k,n}$, giving (\ref{prop-structure-fpa-aiv}).

\medskip

To see that $C_I$ is a $\Cst$\nb-subalgebra, let $p,q,r,s\in P$ such that $\mu(p)=\mu(q)\in I$ and $\mu(r)=\mu(s)\in I$. Then
\[
  w_pw^*_qw_rw^*_s=
  \begin{cases}w_{pq^{-1}(q\vee r)}w^*_{sr^{-1}(q\vee r)}&\text{if $q\vee r<\infty$}\\
  0&\text{else.}
  \end{cases}
\]
Suppose that $w_pw^*_qw_rw^*_s\neq 0$. Then $q\vee r<\infty$. Since $\mu$ is a controlled map  we have 
    \[
    \mu(pq^{-1}(q\vee r))=\mu(p)\mu(q)^{-1}\mu(q\vee r)=\mu(q)\vee \mu(r)=\mu(sr^{-1}(q\vee r)).
    \]
Since $I$ is closed under $\vee$ we have  $\mu(q)\vee \mu(r)\in I$, and  hence \[w_pw^*_qw_rw^*_s=w_{pq^{-1}(q\vee r)}w^*_{sr^{-1}(q\vee r)}\in C_I.\] It follows that  $C_I$ is a $\Cst$\nb-subalgebra. This gives (\ref{prop-structure-fpa-bi}).
\medskip

Let $I\in \Ii$. We prove item (\ref{prop-structure-fpa-bii}) by induction on $|I|$. If $|I|=1$, then $I=\{k\}$ for some $k\in Q$ and $C_I=B_k$.  Now let $n\geq 1$ and assume: if $J\in \Ii$ and $|J|=n$, then $C_J=\sum_{k\in J} B_k$. 

Suppose that $|I|=n+1$.  Let $m$ be a minimal element of $I$ and let $J=I\setminus\{m\}$.  Then $J\in \Ii$ and $|J|=n$. By the induction hypothesis,
\[
C_J=\sum_{k\in J} B_k.
\]
Let $k\in J$.
We claim that $B_kB_m\subseteq C_J$. To see the claim, let $p,q\in \mu^{-1}(k)\cap P$ and $r,s\in \mu^{-1}(m)\cap P$. Then the product of the spanning elements is
\[
w_pw_q^*w_rw_s=\begin{cases}w_{pq^{-1}(q\vee r)}w^*_{sr^{-1}(q\vee r)}&\text{if $q\vee r<\infty$}\\0&\text{otherwise.}\end{cases}
\]
If $q\vee r=\infty$, then $w_pw_q^*w_rw_s=0\in B_k$. So suppose that  $q\vee r<\infty$. Then $k\vee m=\mu(q)\vee\mu(r)<\infty$. Since $m$ is minimal in $I$ and $k\neq m$, we have  $m<k\vee m\in I$. Thus $k\vee m \in J$. Now
\[
\mu\big( pq^{-1}(q\vee r) \big)=\mu(p)\mu(q)^{-1}\big(\mu(q)\vee \mu(r)\big)=\mu(q)\vee \mu(r)= k\vee m\in J.
\]
Similarly, $\mu(sr^{-1}(q\vee r))\in J$. Thus  $w_{pq^{-1}(q\vee r)}w^*_{sr^{-1}(q\vee r)} \in C_J$, and the claim follows.

The claim implies that 
\[C_JB_m=\Big(\sum_{k\in J} B_k\Big) B_m=\sum_{k\in J} B_kB_m\subseteq C_J.
\]
Similarly, $B_m C_J\subseteq C_J$. It follows that $C_J$  is a closed two-sided ideal in the $\Cst$\nb-algebra~$C_I$. Clearly $C_I$ is generated by $C_J$ and the $\Cst$\nb-subalgebra $B_m$ of $C_I$. By \cite[Theorem~3.1.7]{Murphy} we get that $C_I=C_J+B_m$, which proves item (\ref{prop-structure-fpa-bii}).  
\medskip

Finally, for item (\ref{prop-structure-fpa-biii}), we note that $\Ii$ is partially ordered by inclusion and is directed: if $I_1, I_2\in\Ii$, then they are both contained in $I_1\cup I_2\cup (I_1\vee I_2)\in \Ii$. Take a spanning element $w_pw_q^*\in \Cst(G,P)^{\delta_\mu}$. Then $\mu(p)=\mu(q)$, and $w_pw_q^*\in C_I$ where $I=\{\mu(p)\}$. Thus $\Cst(G,P)^{\delta_\mu}=\varinjlim_{I\in\Ii}C_I$.
\end{proof}

The next step is to prove the following lemma from which we can then deduce that the $C_I$'s are nuclear provided that the $B_k$'s are nuclear for all $k\in Q$.

\begin{lem}\label{lem-sum} 
Suppose that $B$ is a nuclear $\Cst$\nb-subalgebra of a $\Cst$\nb-algebra~$C$ and $I\subseteq C$ is a nuclear ideal. Then $B+I$ is a nuclear $\Cst$\nb-algebra. 
\end{lem}
\begin{proof} Notice that $B+I$ is a $\Cst$\nb-subalgebra of $C$ by \cite[Theorem~3.1.7]{Murphy}. The canonical inclusion $I\subseteq B+I$ together with the quotient map $B+I\to (B+I)/I$ produce a short exact sequence $$0\to I\to B+I\to(B+I)/I\to 0.$$ Now $(B+I)/I$ is nuclear because it is isomorphic to $B/(B\cap I)$ and nuclearity passes to quotients. Since~$I$ is nuclear as well and an extension of a nuclear $\Cst$\nb-algebra by a nuclear ideal is again nuclear, we conclude that $B+I$ is also nuclear. 
\end{proof}

\begin{proof}[Proof of Theorem~\ref{thm-nuclearity2}]
By Proposition~\ref{prop-iff-fpa-mu-cutdown}, it suffices to show that the fixed-point algebra \[\Cst(G,P)^{\delta_\mu}=\clsp\{w_pw_q^*: p,q\in P\text{\ and\ }\mu(p)=\mu(q)\}\] is nuclear. To do this, we go back to the analysis of the structure of $\Cst(G,P)^{\delta_\mu}$ in Proposition~\ref{prop-structure-fpa}. Direct limits of nuclear $\Cst$\nb-algebras are nuclear. Since $\Cst(G,P)^{\delta_\mu}=\varinjlim_{I\in\Ii}C_I$, it suffices to show that each $C_I$ is nuclear. We will prove this by induction on $|I|$. For $|I|=1$, we have $C_I=B_k$ for some $k\in I$. Because $B_k$ is a direct limit  $\varinjlim_{n\in\NN}B_{k,n}$, it suffices to show that each $B_{k,n}$ is nuclear. But $B_{k,n}$ is isomorphic to $\Kk(\ell^2(\Lambda_k))\otimes B_e$. Since $B_e$ is nuclear by assumption, each $B_{k,n}$ is nuclear.  Thus $B_k$ is nuclear, as needed. 

Now fix $n\in \NN$, $n>1$ and suppose that $C_{I}$ is nuclear for all $I\in\Ii$ with $|I|=n$. Let $I\in \Ii$ with $|I|=n+1$. Take a minimal element $m\in I$ and consider $J=I\setminus\{m\}$. Then $|J|=n$ and so $C_J$ is nuclear by the induction hypothesis. As in the proof of Proposition~\ref{prop-structure-fpa}, $C_J$ is an ideal in $C_I$ and $C_I=B_m+C_J$. Since $B_m$ is also nuclear, Lemma~\ref{lem-sum} implies that $C_I$ is nuclear as desired.

It follows that $\Cst(G,P)^{\delta_\mu}=\varinjlim_{I\in\Ii}C_I$ is nuclear as well.  Thus $\Cst(G,P)$ is nuclear.

It follows from the nuclearity of $\Cst(G,P)$ that $(G,P)$ is amenable \cite[Theorem~6.42]{Li-book}. (That  \cite[Theorem~6.42]{Li-book} applies was discussed in the proof of Lemma~\ref{lem-subalgebra}.) \end{proof}

\section{Examples}\label{sec:examples}
\subsection{Examples from  \S\ref{sec-control}}

\begin{example}
\label{EE:F}
 Let $\FF$ be the free group on a countable set of generators and let $\FF^+$ be the  unital  free semigroup. Also  let  $\mu\colon\FF\to \ZZ$ be the homomorphism that sends each generator to $1$. We observed in Example~\ref{ex-free-group} that $\mu$ is a controlled map into an abelian group  with $\ker\mu\cap \FF^+=\{e\}$.  Since $\Cst(\ker\mu, \ker\mu \cap \FF^+)\cong\CC$ is nuclear, Theorem~\ref{T:nuclear} implies that $\Cst(\FF, \FF^+)$ is nuclear. (Nica already observed that $(\FF, \FF^+)$ is amenable.)
\end{example}

\begin{example} Consider the subsemigroup $\PP=\{1, bw: w\in\FF^+\}$ of $\FF$ from Example~\ref{ex-scarparo}, and let $\mu\colon\FF\to \ZZ$ be the homomorphism that sends each generator to $1$. Then $\mu$ is a controlled map into an abelian group with $\ker\mu\cap  \PP=\{e\}$. Thus $\Cst(\FF, \PP)$ is nuclear and $(\FF, \PP)$ is amenable by Theorem~\ref{T:nuclear}. 
\end{example}

We already know from \cite[Corollary~A.7]{CaHR} that the  quasi-lattice ordered group $(G,P)$ associated to the Baumslag-Solitar group of Example~\ref{example-BS} is amenable.    We can now use Theorem~\ref{T:nuclear} to see that $\Cst(G,P)$ is nuclear  for all Baumslag-Solitar groups. Together with Proposition~\ref{prop-dilian}, this  shows that Theorem~\ref{T:nuclear} gives new examples of nuclear  semigroup $\Cst$\nb-algebras that do not come from quasi-lattice orders.

\begin{example}
Consider the Baumslag-Solitar groups $G=\langle a, b: ab^c=b^{d}a\rangle$  where either 1) $c,d\geq 1$  or 2) $c\geq 1$ and $d\leq -1$.
Let  $P$ be the semigroup generated by $a$ and $b$.
 We showed in Examples~\ref{example-BS} and~\ref{ex-BS-with-chains} that the height map $\mu\colon G\to \ZZ$ such that $\mu(a)=1$ and $\mu(b)=0$ is a controlled map into an abelian group. Here 
\[
\ker\mu\cap P=\{b^t:t\in\NN\}
\]
is isomorphic to $\NN$. It follows that the subgroup of $G$ generated by $\ker\mu\cap P$ is isomorphic to $\ZZ$ and hence amenable. 
Thus  $\Cst(\ker\mu, \ker\mu \cap P)$ is  nuclear by Corollary~\ref{cor-kernel}. Now Theorem~\ref{T:nuclear} implies that $\Cst(G,P)$ is nuclear and that $(G,P)$ is amenable.
\end{example}

\subsection{HNN extensions, controlled maps and one-relator groups}

Here we consider certain weak quasi-lattice orders arising from HNN extensions of free groups. We were particularly inspired by the analysis of one-relator groups with positive presentation due to Li--Omland--Spielberg \cite{LOS} and the work of an Huef--Raeburn--Tolich on HNN extensions \cite{aHRT} of quasi-lattice orders. Recall that a one-relator group is a group $G=\langle \Sigma; r\rangle$ with generating set~$\Sigma$ and a single defining relator~$r$. We say that the presentation $(\Sigma; r)$ is \emph{positive} if $r$ is a relation of the form $u=v$, where~$u$ and~$v$ are nonempty words in~$\Sigma$. In~\cite{LOS},  Li, Omland and Spielberg analysed when the semigroup $P=\langle \Sigma; u=v\rangle ^+$ defined by the same presentation is right LCM and provided sufficient conditions on the presentation $(\Sigma;u=v)$ for the pair~$(G,P)$ to be a weak quasi-lattice. Under some conditions on the words~$u$ and~$v$,  they found a graph model for the boundary quotient~$\partial \Cst(G, P)$ if $|\Sigma|\geq 3$ and concluded, in particular, that~$\Cst(G, P)$ and $\partial \Cst(G, P)$ are nuclear~ \cite[Corollary~3.10]{LOS}. Here we treat two different classes of presentations $(\Sigma; u=v)$, which are not covered by the results of~\cite{LOS}. If~$|\Sigma|=2$, we obtain nuclearity of the semigroup $\Cst$\nb-algebras of Baumslag--Solitar semigroups as a particular case. More precisely, we will prove the following:

\begin{cor}\label{cor:one-relator-v} Let~$S$ be a nonempty set and let~$\Sigma=S\sqcup\{\sigma\}$. Let~$u$ and $w$ be nonempty words in~$S$ and let~$G=\langle\Sigma; u\sigma=\sigma w \rangle$. Let~$P=\langle \Sigma; u\sigma=\sigma w\rangle^+$ be the semigroup defined by the same relation. Then~$(G,P)$ is an amenable quasi-lattice ordered group and $\Cst(G, P)$ is nuclear.
\end{cor}

\begin{cor}\label{cor:non-Noetherian} Let~$S$ be a nonempty set and let~$\Sigma=S\sqcup\{\sigma\}$. Let~$u$ and $w$ be nonempty words in~$S$ and let~$G=\langle \Sigma; u\sigma w=\sigma  \rangle$. Let~$P=\langle \Sigma; u\sigma w=\sigma \rangle^+$ be the semigroup defined by the same relation. Then~$(G,P)$ is an amenable weakly quasi-lattice ordered group and $\Cst(G, P)$ is nuclear.
\end{cor}

We should observe that the proof of Corollary~\ref{cor:one-relator-v} is essentially done in~\cite{aHRT}. However, we will include it here since our ideas to prove Corollary~\ref{cor:non-Noetherian} are also based on HNN extensions.

We briefly introduce the definition of a Higman--Neumann--Neumann extension (HNN extension, for short). Let~$H$ be a group, and let~$A$ and~$B$ be sugroups of~$H$ with an isomorphism~$\phi\colon A\to B$. The \emph{HNN extension} of~$H$ relative to $A, B$ and $\phi$ is the group $$H^*\coloneqq \langle H, t : t^{-1}at=\phi(a), a\in A\rangle.$$ The group $H$ is the \emph{base} of~$H^*$ and~$t$ is the \emph{stable letter}. We refer the reader to \cite{lyndon1977combinatorial} for further details on this construction. Following~\cite{aHRT}, we let~$\theta\colon H^*\to \ZZ$ be the \emph{height} map. That is, $\theta$ is the group homomorphism that vanishes on~$H$ and maps~$t$ to~$1$. We let~$X$ and~$Y$ be complete sets of left coset representatives for~$H/A$ and~$H/B$, respectively. Every element $h^*\in H^*$ has a unique representation of the form 
\begin{eqnarray}\label{eq:normalform}\begin{aligned}
h^*=h_0t^{\epsilon_1}h_1t^{\epsilon_2}\cdots h_{k-1}t^{\epsilon_k}h_k\qquad(k\geq 0),
\end{aligned}\end{eqnarray} where $\epsilon_i=\pm 1$, $h_{i-1}\in X$ if $\epsilon_i=1$ and $h_{i-1}\in Y$ if~$\epsilon_i=-1$, and~$h_k\in H$. The representation \eqref{eq:normalform} is the (right) \emph{normal form} of~$h^*$ relative to $X$ and $Y$ (see \cite[Theorem~2.1]{lyndon1977combinatorial}).

Let~$\FF$ be the free group on a set of generators~$S$ and let~$u$ and~$w$ be nonempty words in~$S$. We consider the subgroups of~$\FF$ given by~$A\coloneqq \langle u\rangle$ and~$B\coloneqq \langle w\rangle$. Our goal is to analyse the HNN extensions determined by the group isomorphisms $\phi^+,\phi^-\colon A\to B$ given by $u\mapsto w$ and~$u\mapsto w^{-1}$, respectively. We begin with an appropriate choice of left coset representatives for~$\FF/A$ and~$\FF/B$.

\begin{lem}\label{lem:left_cosets} Let~$\FF$ be the free group on a set of generators~$S$ and let~$A =\langle u \rangle$ be the subgroup of~$\FF$ generated by a nonempty word~$u$ in $S$. Define $$\Omega_u\coloneqq\{\alpha\in\FF^+: \alpha=e\text{ or } \alpha \text{ is a nontrivial right divisor of }u\}.$$ 
\begin{enumerate}
\item\label{lem:left_cosets-a}  For $\alpha\in\Omega_u\setminus\{e\}$, let~$X_{\alpha}$ be the set of words $h$ in~$\FF$ for which the largest common suffix with~$u$ in their reduced form is~$\alpha$ so that, in particular,  $h\cdot u$ involves no cancellation. Also, let~$X_e$ be the set of words~$h\in \FF$ for which the products $h\cdot u$ and~$h\cdot u^{-1}$ involve no cancellation. Then $$X\coloneqq\bigsqcup_{\substack{\alpha\in\Omega_u}} X_\alpha$$ is a complete set of left coset representatives for~$\FF/A$.

\item\label{lem:left_cosets-b}   Suppose that $h\in X$ satisfies $hA\cap\FF^+\neq\emptyset$. Then~$h\in\FF^+$ and~$h\leq_{\FF^+}q$ for all~$q\in hA\cap\FF^+$.
\end{enumerate}
\end{lem}

\begin{proof}  We first establish (\ref{lem:left_cosets-a}). Let~$\alpha,\beta\in\Omega_u$. Clearly $X_\alpha\cap X_\beta=\emptyset$ if $\alpha\ne \beta$. Now take~$h\in X_\alpha$ and~$h'\in X_\beta$ and suppose that~$hA=h'A$. Then~$h^{-1}h'\in A$ and we may assume, without loss of generality, that~$$h=h'u^m,\qquad m\geq 0.$$ We claim that~$m=0$. This follows because $h'u$ is in reduced form, and hence there are no cancellations. This implies that  $h$ has $u$ as a suffix, contradicting that $h \in X_\alpha.$
Hence we must have~$m=0$, which gives~$h=h'$.

Now let~$h'\in \FF$ be arbitrary. We will prove that $h'A=hA$ for some~$h\in X$. We can write $h'=h_0 u^m$, where~$m\in \ZZ$, $h_0$ does not end in a nontrivial power of~$u$ and the product~$h_0 \cdot u^m$ is in reduced form. So~$h 'A=h_0A$. If the product~$h_0 \cdot u^{-1}$ is not reduced, then~$h_0\in X_{\alpha}\subseteq X$ for some nontrivial proper right divisor~$\alpha$ of~$u$. Suppose that~$h_0\cdot  u^{-1}$ is reduced. If~$h_0\cdot u $ is also reduced, then~$h_0\in X_e\subseteq X$. Otherwise, i.e. if $h_0u$ is not reduced,  there is a largest proper nontrivial prefix~$\beta$ of~$u$ such that~$\beta^{-1}$ is a suffix of~$h_0$ in its reduced form. Set~$\alpha=\beta^{-1}u$. Let~$h_1\in \FF$ be such that~$h_0=h_1\beta^{-1}$, where the product involves no cancellation. Then~$h_0=h_1(\alpha u^{-1})$ and $h_0A=h_1\alpha A$. Since the product~$h_1\cdot \beta^{-1}$ involves no cancellation, $\alpha$ is indeed the largest common suffix of~$u$ and~$h_1\alpha$. Thus~$h_1\alpha\in X_\alpha$ and~$h'A=h_0A=h_1\alpha A$ as desired.

To prove  (\ref{lem:left_cosets-b}), let~$h\in X$ with~$hA\cap \FF^+\neq\emptyset$. Take~$q\in hA\cap \FF^+$. We will show that~$e\leq h\leq q$ in~$(\FF,\FF^+)$. Let~$m\in \ZZ$ be such that~$h^{-1}q=u^{m}$. We must have~$m\geq 0$ because the product $q\cdot u$ involves no cancellation and~$h$ does not end in a nontrivial power of~$u$ in its reduced form. So~$q=hu^m$ with~$m\geq 0$. Since~$h\in X$, the product~$h\cdot u^m$ involves no cancellation. Therefore, $h\in \FF^+$ and~$h\leq_{\FF^+} q$ as asserted.
\end{proof}

\subsubsection{Case $u\sigma=\sigma w$} We analyse the HNN extension of~$\FF$ relative to $A=\langle u\rangle$, $B=\langle w\rangle$ and the isomorphism $\phi^+\colon A\to B$ that sends~$u$ to~$w$. We begin by showing that~$(\FF^*, \FF^{+*})$ is quasi-lattice ordered by applying Lemma~\ref{lem:left_cosets} and \cite[Theorem~4.1]{aHRT}, which gives the following three sufficient conditions for and HNN-extension $(H^*, P^*)$ as above to be quasi-lattice ordered:
\begin{itemize}
\item[(1)] $\phi(A \cap P) = B \cap P.$
\item[(2)] Every left coset $hA \in H/A$ such that $hA\cap P \neq\emptyset$ has a minimal coset representative $p\in P.$
\item[(3)] for every $x,y \in B$ such that $x \vee y < \infty,$ we have $x \vee y \in B.$
\end{itemize}

\begin{prop}\label{prop:hnn_extension} Let $\FF$ be the free group on a set of generators $S$. Let~$A$ and $B$ be the subgroups of~$\FF$ generated by the nonempty words~$u$ and~$w$ in $S$, respectively. Let $\phi^+\colon A\to B$ be the isomorphism that sends~$u$ to~$w$. Then~$(\FF^*,\FF^{+*})$ is quasi-lattice ordered, where~$\FF^*$ is the HNN extension of~$\FF$ relative to~$A$, $B$ and~$\phi^+$, and~$\FF^{+*}$ denotes the subsemigroup of~$\FF^*$ generated by~$S\cup\{t\}$.

\begin{proof} In order to prove that~$(\FF^*, \FF^{+*})$ is quasi-lattice ordered, we will apply~\cite[Theorem~4.1]{aHRT}. Let~$X$ and~$Y$ be the complete sets of left coset representatives for~$\FF/A$ and~$\FF/B$, respectively, as in Lemma~\ref{lem:left_cosets}. Clearly, $\phi^+(A\cap \FF^+)=B\cap \FF^+$, which is condition (1) of  \cite[Theorem~4.1]{aHRT}. Condition (3) of \cite[Theorem~4.1]{aHRT} is trivially true in our setting. Now suppose that~$h\in X$ satisfies~$hA\cap \FF^+\neq\emptyset$. It follows from Lemma~\ref{lem:left_cosets} that $h\in\FF^+$ and $h\leq_{\FF^+} q$ whenever~$q\in hA\cap \FF^+$. This gives precisely condition (2) of~\cite[Theorem~4.1]{aHRT}. Therefore, $(\FF^+,\FF^{+*})$ is quasi-lattice ordered by \cite[Theorem~4.1]{aHRT}.
\end{proof}
\end{prop}

Before proving the next proposition, we introduce some notation. Let~$H^*$ be the HNN extension of $H$ relative to~$A$, $B$ and $\phi\colon A\to B$. Let $h^*\in H^*$ with normal form~$$h^*=h_0t^{\epsilon_1}h_1t^{\epsilon_2}\ldots h_{k-1}t^{\epsilon_k}h_k\qquad(k\geq 0).$$ The \emph{stem} of~$h^*$, denoted by~$\mathrm{stem}(h^*)$, is the element $h_0t^{\epsilon_1}h_1t^{\epsilon_2}\ldots h_{k-1}t^{\epsilon_k}.$ As in~\cite[Theorem~5.1]{aHRT}, we will use that the height map~$\theta\colon\FF^*\to\ZZ$ is controlled to deduce nuclearity of~$\Cst(\FF^{*},\FF^{+*})$.

\begin{prop}\label{prop:nuclearity_hnn}  Let $(\FF^{*},\FF^{+*})$ be as in Proposition~\textup{\ref{prop:hnn_extension}}. Then the height map $\theta\colon\FF^*\to \ZZ$ is controlled and, in particular, $\Cst(\FF^{*},\FF^{+*})$ is nuclear.
\end{prop}

\begin{proof}  That the height map is controlled is already established in~\cite[Theorem~5.1]{aHRT} and so we will only indicate here how the proof goes.

 Let $x,y\in \FF^{+*}$ be such that~$x\vee y<\infty$. By \cite[Lemma~5.5]{aHRT}, if~$\theta(x)<\theta(y)$, there is~$r\in\FF^+$ with~$x\vee y= \mathrm{stem}(y)r$, and so~$\theta(x\vee y)=\theta(y)=\max\{\theta(x),\theta(y)\}$. If~$\theta(x)=\theta(y)$, then~$\mathrm{stem}(x)=\mathrm{stem}(y)$ and, in this case, $$x\vee y=\mathrm{stem}(y)(p\vee_{\FF^+}q),$$ where~$x=\mathrm{stem}(x)p$ and~$y=\mathrm{stem}(y)q$. Hence~$\theta(x\vee y)=\theta(x)\vee \theta(y)$. Given~$k\in \NN$, the set of minimal elements~$\Sigma_k\subseteq \FF^{+*}$ is then given by~$$\{x\in\theta^{-1}(k)\cap \FF^{+*}: x=\mathrm{stem}(x)\}.$$ It is then easy to see that~$\theta$ is controlled. Since~$\ker\theta\cap\FF^{+*}=\FF^+$ and~$\Cst(\FF, \FF^+)$ is nuclear (see Example \ref{EE:F}), Theorem~\ref{T:nuclear} implies that~$\Cst(\FF^{*},\FF^{+*})$ is nuclear. This completes the proof of the statement.
\end{proof}

\begin{rmk} Although here we are only considering HNN extensions of free groups, a version of \cite[Theorem~5.1]{aHRT} in which the hypothesis ``$(G,P)$ is amenable" is replaced by ``$\Cst(G, P)$ is nuclear" also follows using that the height map is controlled.
\end{rmk}

Rewriting Proposition~\ref{prop:nuclearity_hnn} in terms of one-relator groups with positive presentation gives Corollary~\ref{cor:one-relator-v}:

\begin{proof}[Proof of Corollary \ref{cor:one-relator-v}] Let~$(\FF^*,\FF^{+*})$ be the quasi-lattice ordered group constructed as in Proposition~\ref{prop:hnn_extension}. We obtain an isomorphism $\FF^*\cong G$ by identifying generators $t\mapsto \sigma$, $s\mapsto s$. In particular, this restricts to an isomorphism $\FF^{+*}\cong P$. So $(G,P)$ is quasi-lattice ordered. That $\Cst(G,P)$ is nuclear then follows from Proposition~\ref{prop:nuclearity_hnn}.
\end{proof}

\subsubsection{Case $u\sigma w=\sigma$} We consider the HNN extension of~$\FF$ relative to~$A$, $B$ and the group isomorphism~$\phi^{-}\colon A\to B$ which sends~$u$ to~$w^{-1}$. In this case, we cannot apply \cite[Theorem~4.1]{aHRT} because~$\phi^-(A\cap\FF^+)\not\subseteq B\cap\FF^+$. However, we can directly prove that~$(\FF^*,\FF^{+*})$ is weakly quasi-lattice ordered. We will need the following technical lemma.

\begin{lem}\label{lem:normalform_picture} Let $\FF$ be the free group on~$S$. Let~$A$ and $B$ be the subgroups of~$\FF$ generated by the nonempty words~$u$ and~$w$, respectively, and let $\phi^-\colon A\to B$ be the isomorphism that sends~$u$ to~$w^{-1}$. Let~$\FF^*$ be the corresponding HNN extension and let~$h^*\in\FF^*$ with normal form $$h^*=h_0t^{\epsilon_1}h_1t^{\epsilon_2}\ldots h_{k-1}t^{\epsilon_k}h_k$$ relative to the complete set of representatives~$X$ and $Y$ for~$\FF/A$ and $\FF/B$ as in Lemma~\textup{\ref{lem:left_cosets}}. Then~$h^*\in\FF^{+*}$ if and only if
\begin{enumerate}

\item\label{lem:normalform_picture-a}  $\epsilon_i=1$ for all~$i$, $h_0\in \FF^+$ and~$h_k=w^mq$ for some~$m\in\ZZ$ and~$q\in \FF^+$;

\item\label{lem:normalform_picture-b}  for $2\leq i\leq k$, the double coset $$Bh_{i-1}A=\{w^mh_{i-1}u^n: m,n\in \ZZ\}$$ has a representative in~$\FF^+$.

\end{enumerate}

\begin{proof} Suppose that~$h^*\in\FF^{+*}$. Then there exist~$l\geq0$ and~$p_0,p_1,\ldots, p_l\in\FF^+$ such that~$$h^*=p_0tp_1t\cdots p_{l-1}tp_l.$$ We can compute the normal form of~$h^*$ by working from left to right using its representation~$p_0tp_1t\cdots p_{l-1}tp_l$ and the relation~$ut=tw^{-1}$. No factor~$t^{-1}$ appears in the normal form of~$p_0tp_1t\cdots p_{l-1}tp_l$, from where we deduce that~$l=k$ and~$\epsilon_i=1$ for all~$1\leq i\leq k$.

The representative of the left coset~$p_0A$ is~$h_0\in X$. Since~$p_0\in h_0A\cap \FF^+$, it follows from item (\ref{lem:left_cosets-b}) of Lemma~\ref{lem:left_cosets} that $h_0\in\FF^+$ and $h_0\leq_{\FF^+}p_0$. We write~$p_0=h_0u^{m_0}$, where~$m_0\geq 0$. Next, $h_1\in X$ is the representative of the left coset~$w^{-m_0}p_1$. So there exists~$n_1\in\ZZ$ such that $h_1=w^{-m_0}p_1u^{n_1}$. Then $$w^{m_0}h_1u^{-n_1}=p_1\in Bh_1A\cap\FF^+.$$ Thus the double coset~$Bh_1A$ has representative~$p_1$ in~$\FF^+$. Similarly, we conclude that the double coset~$Bh_{i-1}A$ has a representative in~$\FF^+$ for all~$2\leq i\leq k$. That $h_k$ has the form~$w^{m_k}p_k$ for some~$m_k\in\ZZ$ follows because $u^{-m_k}tp_k=tw^{m_k}p_k$.

Conversely, let $h^*=h_0t^{\epsilon_1}h_1t^{\epsilon_2}\cdots t^{\epsilon_{k-1}}h_{k-1}t^{\epsilon_k}h_k$ and suppose that conditions (\ref{lem:normalform_picture-a}) and (\ref{lem:normalform_picture-b}) in the statement hold. Let~$\theta\colon\FF^*\to\ZZ$ be the height map. Our proof will be by induction on~$\theta(h^*)=k$.

 If~$k=0$, then~$h^*=h_0\in\FF^+\subseteq \FF^{+*}$. If~$k=1$, let~$p_1\in \FF^+$ and~$m_1\in\ZZ$ be such that~$h_1=w^{m_1}p_1$. If~$m_1\geq 0$, we clearly have $$h^*=h_0th_1=h_0tw^{m_1}p_1\in\FF^{+*}.$$ In case~$m_1<0$, the relation~$ut=tw^{-1}$ tells us that~$h^*\in\FF^{+*}$.

Fix~$k\geq 2$ and suppose that~${h'}^{*}\in\FF^{+*}$ whenever its normal form satisfies  (\ref{lem:normalform_picture-a}) and (\ref{lem:normalform_picture-b}), and~$\theta({h'}^{*})\leq k-1$. Let~$m_k\in\ZZ$ and~$p_k\in\FF^+$ be such that~$h_k=w^{m_k}p_k$. Using that the double coset~$Bh_{k-1}A$ has a representative $p_{k-1}\in\FF^+$, we can find~$m_{k-1}, n_k\in\ZZ$ with~$$p_{k-1}=w^{m_{k-1}}h_{k-1}u^{n_{k-1}}.$$ By the induction hypothesis, we conclude that $$h_0th_1t\cdots tw^{-m_{k-1}}p_{k-1}$$ lies in~$\FF^{+*}.$ We claim that~$u^{-n_{k-1}}tw^{m_{k}}$ belongs to~$\FF^{+*}$ as well. Indeed, if~$-n_{k-1},m_k\geq 0$ we are done. Otherwise, we use the relations~$u^{-1}t=tw$ and $ut=tw^{-1}$ to write~$u^{-n_{k-1}}tw^{m_{k}}$ as a product of~$t$ and a nonnegative power of~$u$ or~$w$. Thus one has $u^{-n_{k-1}}tw^{m_{k}}\in\FF^{+*}$ and so $$h^*=h_0th_1t\cdots th_{k-1}th_k=h_0th_1t\cdots tw^{-m_{k-1}}p_{k-1}u^{-n_{k-1}}tw^{m_{k}}p_k$$ lies in~$\FF^{+*}$, as desired.
\end{proof}
\end{lem}


Next we will prove that $(\FF^*, \FF^{+*})$ is weakly quasi-lattice ordered.

\begin{prop}\label{prop:height_map} Let~$\FF^*$ be the HNN extension of~$\FF$ relative to~$A$, $B$ and~$\phi^-$ as in Lemma~\textup{\ref{lem:normalform_picture}}. Then~$(\FF^*,\FF^{+*})$ is weakly quasi-lattice ordered. Let $\theta\colon\FF^*\to \ZZ$ be the height map. If~$x,y\in\FF^{+*}$ and~$x\vee y<\infty$, then~$x$ and~$y$ are comparable and $$\theta(x\vee y)=\max\{\theta(x),\theta(y)\}.$$ 
\end{prop}

\begin{proof} Let~$x,y\in \FF^{+*}$ and suppose that there is a common upper bound in~$\FF^{+*}$ for~$x$ and~$y$.

\textbf{Case 1.} Suppose that $\theta(x)=\theta(y)=:k$. Let~$X$ and $Y$ be the usual sets of left coset representatives for~$\FF/A$ and $\FF/B$ obtained from Lemma~\ref{lem:left_cosets}, respectively. By Lemma~\ref{lem:normalform_picture}, we can write~$x=\mathrm{stem}(x)w^{m}p$ and~$y=\mathrm{stem}(y)w^{n}q$, where~$m,n\in\ZZ$ and~$p,q\in \FF^+$. Without loss of generality, assume $n\geq m$. Since~$x$ and~$y$ have an upper bound in~$\FF^{+*}$, there are~$r,s\in\FF^{+*}$ with~$xr=ys$. When putting $xr$ and $ys$ into normal form using the procedure as in Lemma \ref{lem:normalform_picture}, we work from left to right, and hence do not change $\mathrm{stem}(x)$ or $\mathrm{stem}(y)$. Thus  $\mathrm{stem}(x)=\mathrm{stem}(y)$ by uniqueness of the normal form. 
Applying Lemma~\ref{lem:normalform_picture}~(\ref{lem:normalform_picture-b}) to the normal forms of $xr$ and $ys,$   we also deduce that there are~$r_0, s_0\in\FF^+$ with $$w^{m}p r_0A=w^{n}qs_0A.$$ In particular, $p$ and~$w^{n-m}q$ have an upper bound in~$\FF^+$. This implies that~$p$ and $w^{n-m}q$ are comparable in~$\FF^+$. If $w^{n-m}q\leq_{\FF^+}p$, we have~$$y=\mathrm{stem}(y)w^{n}q=\mathrm{stem}(x)w^{m}w^{n-m}q\leq x.$$ If~$p\leq w^{n-m}q$, then~$$x=\mathrm{stem}(x)w^{m}p=\mathrm{stem}(y)w^{m}p\leq y.$$ So~$x$ and~$y$ are comparable and hence~$x\vee y=\mathrm{max}\{x,y\}$. Thus~$\theta(x\vee y)=\theta(x)=\theta(y)$.

\textbf{Case 2.} Suppose that~$\theta(x)\ne \theta(y)$. Without loss of generality,  we assume that $\theta(x)<\theta(y)$. We shall prove that~$x<y$. Let $$y=h_0t\cdots th_{\theta(x)}t\cdots th_{\theta(y)}$$ be the normal form of~$y$ relative to the complete sets of left coset representatives~$X$ and~$Y$. Since there exists an upper bound in~$\FF^{+*}$ for~$x$ and~$y$, it follows that 
$$
y=\mathrm{stem}(x)  h_{\theta(x)}t\dots th_{\theta(y)}.
$$ 
 Let~$m\in\ZZ$ and~$p\in \FF^+$ be such that~$x=\mathrm{stem}(x)w^mp$. Since there exist~$r,s\in\FF^{+*}$ such that~$xr=ys,$ we can find~$r_0\in \FF^+\cap X$ with~$w^mpr_0A=h_{\theta(x)}A$.  Then there is~$n\in \ZZ$ with~$w^mpr_0=h_{\theta(x)}u^n$. Thus 
\[
 x^{-1}y
 =r_0u^{-n}h_{\theta(x)}^{-1}\mathrm{stem}(x)^{-1}y
 =r_0tw^nh_{\theta(x)+1}t\cdots th_{\theta(y)}.
\]
 If $\theta(y)=\theta(x)+1$, we have~$x^{-1}y=r_0tw^nh_{\theta(x)+1}\in\FF^{+*}$ by Lemma~\ref{lem:normalform_picture}, because $y\in\FF^{+*}$. Suppose $\theta(x)+1<\theta(y)$. Observe that the double coset $Bw^nh_{\theta(x)+1}A$ coincides with $Bh_{\theta(x)+1}A$. On the other hand, $Bw^nh_{\theta(x)+1}A$ equals the double coset of the representative of the left coset~$w^nh_{\theta(x)+1}A$. It follows that this latter double coset has a representative in~$\FF^+$ because $Bh_{\theta(x)+1}A$ does so. Proceeding with this argument for all $\theta(x)+1\leq i <\theta(y)$ and using Lemma~\ref{lem:normalform_picture} again, we deduce that $x^{-1}y\in\FF^{+*}$. So $x<y$. Thus $y=x\vee y$ and $\theta(x\vee y)=\theta(y)=\max\{\theta(x),\theta(y)\}$.
\end{proof}

\begin{prop} 
Let~$(\FF^*,\FF^{+*})$ be as in Proposition~\textup{\ref{prop:height_map}}. Then the height map $\theta\colon \FF^*\to\ZZ$ is controlled and, in particular, $\Cst(\FF^{*},\FF^{+*})$ is nuclear. 
\end{prop}

\begin{proof} By Proposition~\ref{prop:height_map}, we have~$\theta(x\vee y)=\theta(x)\vee\theta(y)$ whenever~$x,y\in\FF^{+*}$ have an upper bound in~$\FF^{+*}$. We will show that~$\theta$ also satisfies item (b) of Definition~\ref{defn-less-controlled}. 

Let~$q\in \NN$. If~$q=0$, we take~$\Lambda_0=\{e\}$. So~$$S_e\coloneqq\theta^{-1}(0)\cap \FF^{+*}=\{x\in\theta^{-1}(0)\cap\FF^{+*}: x\geq e\}=\FF^+$$ and the required sequence $(s_n^e)_{n\in\NN}\subseteq S_e$ is just the constant sequence $s_n^e=e$ for all~$n\in\NN$.

Assume~$q\geq 1$. Set~$${}_w\!X\coloneqq\{h\in X: BhA=BpA~\text{for some } p\in\FF^+\}.$$ Put~$\Lambda_q=(X\cap \FF^+)\times ({}_w\! X)^{q-1}$ (in case~$q=1$, we naturally identify $({}_w\! X) ^{q-1}$ with the unit set). For~$\lambda=(h_0, h_1, \ldots, h_{q-1})\in \Lambda_q$, we let $$S_{\lambda}\coloneqq\{x\in\theta^{-1}(q)\cap\FF^{+*}: \mathrm{stem}(x)=h_0th_1t\cdots th_{q-1}t\}.$$ Then $$\theta^{-1}(q)\cap\FF^{+*}=\bigsqcup_{\substack{\lambda\in\Lambda_q}}S_\lambda.$$ The required decreasing sequence associated to~$\lambda=(h_0, h_1, \ldots, h_{q-1})$ is the sequence whose $n$-th term is $$s_n^{\lambda}\coloneqq h_0t\cdots th_{q-1}tw^{-n}.$$ We claim that $$S_{\lambda}=\bigcup_{n\in\NN}\{x\in\theta^{-1}(q)\cap \FF^{+*}:x\geq s_n^{\lambda}\}.$$ Indeed, given $x\in S_{\lambda}$, there are~$m\in\ZZ$ and~$p\in\FF^+$ with $x=\mathrm{stem}(x)w^mp$ by Lemma~\ref{lem:normalform_picture}. Take~$n=|m|$. Then $$(s_n^\lambda)^{-1}x=w^{|m|}w^mp\in\FF^+\subseteq\FF^{+*}.$$ This proves our claim. We conclude that~$\theta$ is controlled in the sense of Definition~\ref{defn-less-controlled}. 

Because~$\ker\theta\cap\FF^{+*}=\FF^+$ and~$\Cst(\FF, \FF^+)$ is nuclear, we deduce from Theorem~\ref{T:nuclear} that~$\Cst(\FF^{*},\FF^{+*})$ is also nuclear.
\end{proof}

\begin{rmk} Observe that the height map $\theta\colon \FF^*\to\ZZ$ above is not controlled in the sense of Definition~\ref{def-controlled}. Indeed, the infinite descending sequence $$\cdots\leq t w^{-(n+1)}\leq tw^{-n}\leq \cdots$$ is not bounded below by any element~$x\in\theta^{-1}(1)\cap\FF^{+*}$.
\end{rmk}
 The proof of Corollary \ref{cor:non-Noetherian} now follows as in Corollary \ref{cor:one-relator-v}.

\subsection{Semidirect products of weakly quasi-lattice ordered groups}

In this subsection, we analyse the semidirect product of weak quasi-lattice orders. We provide examples in which the semidirect product is again weakly quasi-lattice ordered with a natural positive cone, and study nuclearity of the corresponding semigroup $\Cst$\nb-algebra using controlled maps. Most of our examples involve free groups.

 Let $(N, L)$ and $(H,Q)$ be weak quasi-lattice orders. Suppose that $\varphi\colon H \leq \mathrm{Aut}(N)$ is a subgroup of automorphisms of~$N$, so that we can form the semidirect product
$$G = N \rtimes_\varphi H,$$ where $\varphi\colon H\to\mathrm{Aut}(N),$ $h\mapsto\varphi_h$ is the corresponding action of~$H$ on~$N$. Recall how the
semidirect product is defined: the underlying set is the Cartesian product $N\times H$ of~$N$ and~$H$, and multiplication is given by
 $$(n,h)(n',h')\coloneqq (n\varphi_h(n'), hh').$$

Our first goal is to find sufficient conditions for the semidirect product $G=N\rtimes_{\varphi}H$ to be weakly quasi-lattice ordered as well, with positive cone $P=L\rtimes_\varphi Q$.
 
\begin{lem} Let $(N, L)$ and $(H,Q)$ be as above and suppose that for all $q \in Q$
$$\varphi_q(L) \subseteq L.$$
Then 
$$P = L \rtimes_\varphi Q$$
is a generating subsemigroup of~$G$ such that $P\cap P^{-1} = \{e\}.$
\end{lem}

In this case we have a partial order on $G$ given by $g_1\leq_P g_2$ if and only if $g_1^{-1}g_2\in P$, and we write:
$$(G,P) = (N,L)\rtimes_\varphi (H,Q).$$
\begin{proof} That $P$ generates $G$ follows directly from the construction. Now suppose $(l,q) \in P\cap P^{-1}.$ Since $(l,q)^{-1}=(\varphi_{q^{-1}}(l^{-1}), q^{-1})$ and $Q\cap Q^{-1}=\{e_q\}$, we must have $q=e_Q$. Hence $\varphi_{q^{-1}}(l^{-1})=l^{-1}\in L\cap L^{-1}$ and so $l=e_L$.
\end{proof}

We will invoke the following fact several times in this subsection.

\begin{lem}\label{lem-leq1}
Let $(N,L)$ and $(H,Q)$ be weak quasi-lattices with $H \leq \mathrm{Aut}(N)$ such that $\varphi_q(L)\subseteq L$ for all $q \in Q$, and let $(G,P) = (N,L)\rtimes_\varphi (H,Q)$ be the semidirect product. If $(l,q) \leq_P (m,r)$ with $l,m \in L$ and $q,r \in Q$, then $l\leq_L m$ and $q \leq_Q r$. 
\end{lem}

\begin{proof} Assume that $(l,q) \leq_P (m,r)$. This entails
$$(l,q)^{-1}(m,r) = (\varphi_{q^{-1}}(l^{-1})\varphi_{q^{-1}}(m), q^{-1}r) \in (L,Q).$$
Hence $q \leq_Q r.$ Also, $\varphi_{q^{-1}}(l^{-1})\varphi_{q^{-1}}(m)= \varphi_{q^{-1}}(l^{-1}m) \in L$, which implies that $l^{-1}m \in \varphi_q(L) \subseteq L.$ Hence $l \leq_L m.$
\end{proof}

\begin{lem}\label{lem-leq2}
Let $(N,L)$ and $(H,Q)$ be weak quasi-lattices with $H \leq \mathrm{Aut}(N)$. Suppose that $\varphi_q(L)=L$ for all $q \in Q.$ Then the semidirect product
$$(G,P) = (N,L) \rtimes_\varphi (H,Q)$$ is weakly quasi-lattice ordered. In this case, $$(l,q)\vee_P(m,r)=(l \vee_L m, q\vee_Q r).$$
\end{lem}
\begin{proof} Suppose that $(l,q), (m,r)\in P$ have a common upper bound. By Lemma \ref{lem-leq1} and the fact that $(N,L)$ and $(H,Q)$ are weak quasi-lattices, we have $l \vee_L m <\infty$ and $q\vee_Q r<\infty.$ We claim that $(l \vee_L m, q\vee_Q r)\in P$ is the least upper bound of $(l,q)$ and $(m,r)$. Indeed, we have that \begin{equation*}\begin{aligned}(l,q)^{-1}(l \vee_L m, q\vee_Q r)&=(\varphi_{q^{-1}}(l^{-1}),q^{-1})(l \vee_L m, q\vee_Q r)\\&=(\varphi_{q^{-1}}(l^{-1}(l \vee_L m)), q^{-1}(q\vee_Q r)).
\end{aligned}\end{equation*} This belongs to~$L$ because $l\leq_L l\vee_Lm$, $q\leq_Q q\vee_Qr$ and $\varphi_{q^{-1}}(L)=\varphi_{q^{-1}}(\varphi_q(L))=L$. Similarly, $(m,r)\leq_P (l \vee_L m, q\vee_Q r)$.

Now let $(n,s)\in P$ be a common upper bound for $(l,q)$ and $(m,r)$. All we have to show is that $(l \vee_L m, q\vee_Q r)\leq_P(n,s)$. By Lemma \ref{lem-leq1}, we must have $l \vee_L m\leq_L n$ and $q\vee_Q r\leq_Q s$. Using that $\varphi_{(q\vee_Q r)^{-1}}(L)=L$, we obtain $(l \vee_L m, q\vee_Q r)\leq_P(n,s)$. So $(l \vee_L m, q\vee_Q r)$ is the least upper bound of~$(l,q)$ and $(m,r)$ as claimed.
\end{proof}

The assumption $\varphi_q(L)=L$ for all~$q\in Q$ is fairly strong, as we will see later in examples. However, in the situation of the previous lemma, we may use a natural controlled map to deduce nuclearity of~$\Cst(N\rtimes_\varphi H, L\rtimes_{\varphi}Q)$.

\begin{prop}\label{semi-controlled} Let $(N,L)$ and $(H,Q)$ be weak quasi-lattices with~$H \leq \mathrm{Aut}(N)$. Let $$\begin{array}{lccc} \mu\colon & N \rtimes_\varphi H& \rightarrow &H \\
				 &(n,h) &\mapsto & h
				 \end{array} $$ be the natural surjection. Suppose that $x,y\in P=L\rtimes_{\varphi}Q$ have a least upper bound $x\vee_P y$ in $P$. Then $\mu(x\vee_P y)=\mu(x)\vee_Q\mu(y)$. If $\varphi_q(L)=L$ for all $q \in Q,$ then $\mu$
is controlled in the sense of Definition \textup{\ref{def-controlled}}.
\end{prop}

\begin{proof}Let $(G,P) = (N,L) \rtimes (H,Q)$. Let $(l,q),(m,r) \in P$ be such that $(l,q) \vee_P (m,r)$ exists. Write $(l,q) \vee_P (m,r)=(n,s)$. In order to prove the first part of the statement, it suffices to show that $s=q\vee_Qr$. We have that $q\vee_Qr$ exists and $q,r\leq_Qq\vee_Qr\leq_Qs$. Also, both $\varphi_{q^{-1}}(l^{-1}n)$ and $\varphi_{r^{-1}}(m^{-1}n)$ belong to $L$. Hence $(n,q\vee_Qr)$ is a common upper bound for $(l,q)$ and $(m,r)$, which gives $(n,s)\leq_P(n,q\vee_Qr)$. Thus $s\leq_Q q\vee_Qr$ and so $s=q\vee_Q r$ as desired. 

Suppose that $\varphi_q(L)=L$ for all $q \in Q.$ So $(G,P)$ is weakly quasi-lattice ordered by Lemma~\ref{lem-leq2}. We have proven that $\mu$ satisfies item \eqref{cm-1-new} of Definition~\ref{def-controlled}. In order to show that $\mu$ also satisfies  \eqref{cm-2-new}, we observe that for all~$q\in Q$,
$$\mu^{-1}(q)\cap P = \{(l,q): l\in L\}.$$ So the required set of minimal elements is $\Sigma_q=\{(e,q)\}$, where $e$ denotes the unit element of~$N$.
This verifies item \eqref{cm-2-new} of Definition~\ref{def-controlled}.
\end{proof}

We now consider examples (and non-examples) coming from free-by-cyclic groups $\FF \rtimes \ZZ$. 

\begin{example} Let $\FF = \langle a,b \rangle$ be the free group on two generators. We let~$\ZZ$ act on $\FF$ via the automorphism $\varphi\colon\FF\to \FF$ given by $\varphi(a)=b$ and $\varphi(b)=a$. In this case we have $\varphi(\FF^+) = \FF^+$ and so $(\FF,\FF^+)\rtimes_\varphi(\ZZ,\NN)$ is a weak quasi-lattice. By Proposition~\ref{semi-controlled}, the canonical projection $\mu\colon\FF\rtimes_\varphi\ZZ\to\ZZ$ is controlled. Since~$(\ker\mu)\cap(\FF^+\rtimes_\varphi\NN)=\FF^+$, Theorem~\ref{thm-nuclearity2} tells us that $\Cst(\FF\rtimes_\varphi\ZZ,\FF^+\rtimes_\varphi\NN)$ is nuclear.
\end{example}

\begin{example} More generally, if $\FF$ is the free group on $n$ generators and $\varphi$ acts by permuting generators, then~$\varphi(\FF^+)=\FF^+$ and $(\FF,\FF^+)\rtimes_\varphi(\ZZ,\NN)$ is weakly quasi-lattice ordered. By the same reasoning as above we conclude that $\Cst(\FF\rtimes_\varphi\ZZ,\FF^+\rtimes_\varphi\NN)$ is nuclear.
\end{example}

Next, we give an example of a free-by-cyclic group $\FF \rtimes_{\varphi} \ZZ$ with $\varphi(\FF^+)\subsetneq\FF^+$ and such that $(\FF,\FF^+)\rtimes_\varphi(\ZZ,\NN)$  is not a weak quasi-lattice order. 

\begin{example}\label{ex:nonexample} Let $\FF$ be freely generated by~$a$ and~$b$, and let~$\varphi$ be defined by $\varphi(a)=ba$, $\varphi(b)=b^2a$. Then $(\FF,\FF^+)\rtimes_\varphi(\ZZ,\NN)$  is not weakly quasi-lattice ordered. To see this, notice that $\varphi^2$ is the automorphism of~$\FF$ given by $$\varphi^2(a)=b^2aba,\qquad\varphi^2(b)=b^2ab^2aba.$$ Take $p=(a,2)$ and $q=(ab,1)$. Thus $p,q\in P=(\FF^+,\NN)$ have $n_1=(ab^2aba,2)$ and $n_2=(ab^2ab^2aba,2)$ as common upper bounds in~$P$. To see that $p$ and $q$ have no least upper bound, suppose that~$(s,m)$ is less than $n_1$ and $n_2$ in~$P$ and greater than $p,q$. Then $m=2$ and $ab\leq_{\FF^+}s\leq_{\FF^+}ab^2ab$ by Lemma~\ref{lem-leq1}. In particular, $s^{-1}(ab^2aba)\not\in \varphi^2(\FF^+)$, contradicting the assumption that~$(s,m)\leq_P n_1$.
\end{example}

In view of Example~\ref{ex:nonexample}, we now present some sufficient conditions for a semidirect product $(\FF,\FF^+)\rtimes_{\varphi}(H,Q)$ to be weakly quasi-lattice ordered.

\begin{prop}\label{prop:sufficient-cond}
Let $\FF$ be a free group on the set of generators $S$ and let $(\FF, \FF^+)$ be the quasi-lattice ordered group obtained from the free semigroup~$\FF^+$ on~$S$. Let $(H,Q)$ be a weak quasi-lattice with~$H \leq \mathrm{Aut}(\FF)$ such that if $q,r\in Q$ and~$q\vee_Q r<\infty$, then $q$ and $r$ are comparable. Suppose further that, for all~$q\in Q$, we have $\varphi_q(\FF^+) \subseteq \FF^+$  and that $\varphi_q(x)$  and $\varphi_q(y)$ have no common divisors whenever~$x,y\in S$ are different generators of~$\FF$. Then 
$$(\FF, \FF^+) \rtimes_\varphi (H,Q)$$
is a weak quasi-lattice. 
\end{prop}

\begin{proof} Let $l,m\in\FF^+$, $q,r\in Q$ be such that~$(l,q)$ and~$(m,r)$ have a common upper bound in~$(\FF^+,Q)$. In particular, $q\vee_Q r<\infty$ and $l\vee_{\FF^+}m<\infty$. We may suppose, without loss of generality, that~$l$ is a prefix of~$m$. Say $(l,q), (m,r)\leq_P(n,s)$. Then $$(\varphi_{q^{-1}}(l^{-1}n), q^{-1}s), (\varphi_{r^{-1}}(m^{-1}n), r^{-1}s)\in P.$$ Thus there exist $v,v'\in\FF^+$ such that $$\varphi_q(v)=l^{-1}n,\qquad \varphi_r(v')=m^{-1}n,$$ which gives $l\varphi_q(v)=m\varphi_r(v')$. It follows that the set $$W\coloneqq\{\alpha\in\varphi_r(\FF^+): l^{-1}m\alpha\in\varphi_q(\FF^+)\}$$ is nonempty and bounded below by~$e$. We claim that~$W$ has a unique minimal element. Indeed, if~$l^{-1}m\in\varphi_q(\FF^+)$, then~$e$ is the unique minimal element of~$W$. Now suppose that~$l^{-1}m\not\in\varphi_q(\FF^+)$ and let~$w,w'$ be minimal elements of~$W$. Let~$p,p'\in \FF^+$ with~$\varphi_q(p)=l^{-1}mw$ and~$\varphi_q(p')=l^{-1}mw'$. Since~$l^{-1}m$ is a prefix of both~$\varphi_q(p)$ and $\varphi_q(p')$, our assumption that different generators have relatively prime images under~$\varphi_q$ implies that the first letters of~$p$ and~$p'$ coincide. Suppose that~$s_1$ is the first letter of~$p$ and~$p'$. If~$\varphi_q(s_1)<_{\FF^+}l^{-1}m$, then the first letters of~$s_1^{-1}p$ and $s_1^{-1}p'$ must coincide. Proceeding with this argument, we can find a largest prefix~$u$ of~$p$ and~$p'$ so that~$\varphi_q(u)<_{\FF^+}l^{-1}m$. Again the first letter of~$u^{-1}p$, say~$t$, must coincide with the first letter of~$u^{-1}p'$ because $\varphi_q(u^{-1})l^{-1}m$ is a nontrivial prefix of both $\varphi_q(u^{-1}p)$ and $\varphi_q(u^{-1}p')$. We also have~$l^{-1}m<_{\FF^+}\varphi_q(ut)$ because of our assumption that $l^{-1}m\not\in\varphi_q(\FF^+)$. Now~$(l^{-1}m)^{-1}\varphi_q(ut)$ is a prefix of both~$w$ and~$w'$. Using that different elements of~$S$ have relatively prime images under~$\varphi_r$, we deduce that there are~$t_1, v, v'\in\FF^+$ with $w=\varphi_r(t_1v)$, $w'=\varphi_r(t_1v')$. If~$\varphi_r(t_1)=(l^{-1}m)^{-1}\varphi_q(ut)$, then~$v=v'=e$ and~$w=w'$. Otherwise, we have either~$\varphi_r(t_1)<_{\FF^+}(l^{-1}m)^{-1}\varphi_q(ut)$ or~$(l^{-1}m)^{-1}\varphi_q(ut)<_{\FF^+}\varphi_r(t_1)$. The latter inequality requires the first letter of~$v$ be equal to the first letter of~$v'$, while the former inequality implies that the first letters of~$(ut)^{-1}p$ and~$(ut)^{-1}p'$ coincide. Employing this argument finitely many times, we will arrive at~$p=p'$ or $v=v'$. Therefore~$w=w'$ and this proves our claim that~$W$ has a unique minimal element.

Finally, we prove that~$(mw, q\vee r)$ is the least upper bound of~$(l,q)$ and~$(m,r)$. Let~$(n,s)$ be an upper bound for~$(l,q)$ and~$(m,r)$. Then~$s\geq q\vee r$. Moreover, $l^{-1}n=l^{-1}mm^{-1}n\in\varphi_q(\FF^+)$ and~$m^{-1}n\in\varphi_r(\FF^+)$. Hence~$m^{-1}n\in W$ and it follows that $w^{-1}(m^{-1}n)\in \FF^+$. Let~$p'\in\FF^+$ be such that~$m^{-1}n=\varphi_r(p')$. Thus $\varphi_r(p^{-1}p')\in\FF^+$. However, using our hypothesis again we deduce that $\varphi_r(p^{-1}p')\in\FF^+$ if and only if $p^{-1}p'\in\FF^+$. So~$(mw)^{-1}n=w^{-1}m^{-1}n\in\varphi_r(\FF^+)$. By a similar argument, $(mw)^{-1}n\in\varphi_q(\FF^+)$. Here we invoke the fact that  if~$q\vee r<\infty$ then $q\vee r=q$ or $q\vee r = r$ to guarantee that $(mw, q\vee r)\leq (n,s)$ in $(\FF, \FF^+) \rtimes_\varphi (H,Q)$.
\end{proof}

We will now illustrate how Proposition \ref{prop:sufficient-cond} can be used to produce examples of weak quasi-lattices as well as controlled maps.

\begin{example}\label{ex:aut-free-group} Let $\varphi\colon\FF\to \FF$ be given by $\varphi(a)=ab$ and $\varphi(b)=b$. Here $a \notin \varphi(\FF^+),$ and so Lemma \ref{lem-leq2} does not apply to this case. Nevertheless, $(\FF,\FF^+)\rtimes_\varphi (\ZZ,\NN)$ is still a weak quasi-lattice by Proposition~\ref{prop:sufficient-cond}, because $\varphi(a)=ab$ and $\varphi(b)=b$ are relatively prime. 

We will prove that $\Cst(\FF^+\rtimes_\varphi\NN)$ is nuclear using Theorem~\ref{thm-nuclearity2}. To do so, consider the homomorphism $\sigma_a\colon\FF\to\ZZ$ that sends $a$ to~$1$ and $b$ to $0$. Since $\varphi$ does not change the exponent sum of~$a$ in a word, this induces a homomorphism~$\mu\colon\FF\rtimes_{\varphi}\ZZ\to\ZZ^2$ defined by $\mu(w,n)\coloneqq (\sigma_a(w),n)$. We will see that $$\mu\colon (\FF\rtimes_\varphi\ZZ, \FF^+\rtimes_\varphi\NN)\to(\ZZ^2,\NN^2)$$ is controlled. 

We begin by showing that $\mu(x\vee_Py)=\mu(x)\vee\mu(y)$ when it exists. Let~$x=(p,m)$ and $y=(q,n)$ with a common bound $(s,l)$ in~$P=\FF^+\rtimes_\varphi\NN$. In particular, $p$ and~$q$ have a common upper bound in~$\FF^+$ and we may suppose, without loss of generality, that~$p\leq_{\FF^+}q$. That is, $p$ is a prefix of~$q$. Since~$q\leq_{\FF^+}s$ we have $p^{-1}s=p^{-1}qq^{-1}s\in \varphi_m(\FF^+)$. Thus $p^{-1}q$ has reduced form $$b^{i_0}ab^{i_1}a\cdots ab^{i_k},\qquad(k\geq 0)$$ where $i_0, i_k\in \NN$ and~$i_j\geq  m$ for all~$j\neq 0, k$. It follows that the least upper bound for~$(p,m)$ and $(q,n)$ is given by $$(p,m)\vee_P(q,n)=\begin{cases}(q, m\vee_{\NN}n) ,& \text{if } k=0 \text{ or } i_k\geq m \\
(qb^{m-i_k},m\vee_{\NN}n) & \text{else}.\end{cases}$$ From this we deduce that $\mu(x\vee_Py)=\mu(x)\vee\mu(y)$, which is item (\ref{cm-1-new}) of Definition~\ref{def-controlled}.

To see that $\mu$ satisfies item \eqref{cm-2-new} of Definition~\ref{def-controlled}, take $(m,n)\in \NN^2$. Observe that~$\sigma_a(p)$ is precisely the exponent sum of~$a$ in~$p\in\FF$. Then $\mu^{-1}(m,n)\cap P$ consists of all elements of the form $(p,n)\in\FF^+\rtimes_\varphi\NN$ for which the exponent sum $\sigma_a(p)$ is~$m$. Set $$\Sigma_{(m,n)}\coloneqq\{(p,n)\in\FF^+\rtimes_\varphi \NN : \sigma_a(p)=m\text{ and }b \text{ is not a suffix of }p\}.$$ Since $b=\varphi_n^{-1}(b)\in\varphi^{-1}_n(\FF^+)$, every $(q,n)\in\mu^{-1}(m,n)\cap P$ can be written as the product $(p,n)\cdot (b^k,0)$ for some~$(p,n)\in\Sigma_{(m,n)}$ and~$k\in\NN$ . So $(p,n)\leq_P(q,n)$. Moreover, if $(p,n), (q,n)\in\Sigma_{(m,n)}$ and $p\neq q$, then $p\vee_{\FF^+}q=\infty$, because~$\sigma_a(p^{-1}q)=0$ and both~$p$ and~$q$ end in~$a$ and so cannot be prefixes of one another. Hence $(p,n)\vee_P(q,n)=\infty$ and we conclude that~$\mu$ is controlled. 

Since~$(\ker\mu)\cap(\FF^+\rtimes_\varphi\NN)=\langle b\rangle\times\{0\}$, it follows from Theorem~\ref{thm-nuclearity2} that~$\Cst(\FF\rtimes_\varphi\ZZ,\FF^+\rtimes_\varphi\NN)$ is nuclear.

\end{example}

\begin{example} Let $c,d\geq 1$ such that $c+d$ is even. We will construct an action of the Baumslag--Solitar group $\mathrm{BS}(c,-d)=\langle a,b\colon ab^c=b^{-d}a\rangle$ on the free group~$\FF$ with generators~$\{x,y,z\}$.

Let~$\varphi_a, \varphi_b\colon\FF\to\FF$ be given by \begin{equation*}\begin{aligned}\varphi_a(x)&=xy,&\qquad\varphi_a(y)&=y,&\qquad\varphi_a(z)&=zy;&\\\varphi_b(x)&=z,&\qquad\varphi_b(y)&=y,&\qquad\varphi_b(z)&=x.&\end{aligned}\end{equation*} Then $\varphi_a,\varphi_b$ are commuting automorphisms of~$\FF$, and $\varphi^2_b=\id_{\FF}$. Since~$c+d$ is even, $\varphi_b^{c+d}=\id_{\FF}$. Thus~$\varphi_a, \varphi_b$ induce an action $$\varphi\colon\mathrm{BS}(c,-d)\to\mathrm{Aut}(\FF)$$ satisfying $\varphi_p(\FF^+)\subseteq\FF^+$ for all $p\in\mathrm{BS}(c,-d)^+$. Notice that $\varphi_p(x),\varphi_p(y)$ and $\varphi_p(z)$ are relatively prime. Hence $$(\FF,\FF^+)\rtimes_\varphi(\mathrm{BS}(c,-d),\mathrm{BS}(c,-d)^+)$$ is weakly quasi-lattice ordered by Proposition~\ref{prop:sufficient-cond}.

We will now show that $\Cst(\FF^+\rtimes_\varphi\mathrm{BS}(c,-d)^+)$ is nuclear. Consider the homomorphism $\sigma\colon\FF\to\ZZ$ that sends both~$x$ and $z$ to~$1$ and~$y$ to~$0$. Let $\theta\colon\mathrm{BS}(c,-d)\to\ZZ$ be the height map. Set $\mu(u,w)\coloneqq(\sigma(u),\theta(w))$. Then $$\mu\colon\FF\rtimes_\varphi\mathrm{BS}(c,-d)\to\ZZ^2$$ is a group homomorphism because $\varphi$ does not change the exponent sum $\sigma(u)=\sigma_x(u)+\sigma_z(u)$ in a word~$u$. We claim that~$\mu$ is controlled. Let $(r,p), (s,q)\in P=\FF^+\rtimes_\varphi\mathrm{BS}(c,-d)^+$ with a common upper bound in~$P$. We must have $r\vee_{\FF^+}s<\infty$. Without loss of generality, assume that $r$ is a prefix of~$s$. If~$\theta(p)=0$, then $(r,p)\vee_P(s,q)=(s, p\vee q)$ and so $$\mu((r,p)\vee_P(s,q))=\mu(r,p)\vee_{\ZZ^2}\mu(s,q).$$ In case $\theta(p)>0$, the same reasoning as in Example~\ref{ex:aut-free-group} gives a least $m\in \NN$ so that $r^{-1}sy^m\in\varphi_p(\FF^+)$. Then $(r,p)\vee_P(s,q)=(sy^m, p\vee q)$ and so $$\mu((r,p)\vee_P(s,q))=\mu(r,p)\vee_{\ZZ^2}\mu(s,q).$$

It remains to show that~$\mu$ satisfies item \eqref{cm-2-new} of Definition~\ref{defn-less-controlled}. First, for $m\in\NN$, we set $$\Sigma_{m}\coloneqq\{r\in\FF^+:\sigma(r)=m\text{ and }y \text{ is not a suffix of }r\}.$$ Then $r\vee_{\FF^+}s=\infty$ whenever $r, s\in \Sigma_{m}$ and $r\neq s$. Second, because the height map $\theta\colon \mathrm{BS}(c,-d)\to\ZZ$ is controlled, for all $k\in\NN$, there exist a set~$\Lambda_{k}$ and a decreasing sequence $(s_n^{\lambda})_{n\in\NN}\subseteq\mu^{-1}(k)\cap\mathrm{BS}(c,-d)^+$ associated to each~$\lambda\in\Lambda_k$ so that the conditions of Definition~\ref{defn-less-controlled} hold. Finally, for each $(m,k)\in\NN^2$ we set $$\Lambda_{(m,k)}\coloneqq\Sigma_m\times\Lambda_k.$$ Thus $\gamma_1,\gamma_2\in\Lambda_{(m,k)}$ and $\gamma_1\neq \gamma_2$ entails $\gamma_1\vee_P\gamma_2=\infty$. The decreasing sequence associated to~$(r,\lambda)\in\Sigma_m\times\Lambda_k$ has $n$-th term~$(r,s_n^{\lambda})$. Given $(s,p) \in P$ with~$\mu(s,p)=(m,k)$, write $s=ry^l$, where~$r\in\FF^+$ ends in~$x$ or~$z$ and~$l\in\NN$. Because~$y=\varphi_q(y)\in\varphi_q(\FF^+)$ for all~$q\in\mathrm{BS}(c,-d)^+$, it follows that there are~$n\in \NN$ and $\lambda\in\Lambda_k$ such that $(r,s^\lambda_n)\leq_{P}(s,p)$. This completes the proof that $\mu$ also satisfies item \eqref{cm-2-new} of Definition~\ref{defn-less-controlled} and, therefore, is a controlled map. 

Now $(\ker\mu)\cap(\FF^+\rtimes_{\varphi}\mathrm{BS}(c,-d)^+)=\langle y\rangle\times \langle b\rangle$. Hence $\Cst(\FF\rtimes_{\varphi}\mathrm{BS}(c,-d), \FF^+\rtimes_{\varphi}\mathrm{BS}(c,-d)^+)$ is nuclear by Theorem~\ref{T:nuclear}.
\end{example}

We finish this  section by presenting an example with a slightly different flavour. In particular, Lemma~\ref{lem-leq2} and Proposition~\ref{prop:sufficient-cond} do not apply in this case and we have to show the existence of least upper bounds directly.

\begin{example}Let $\FF = \langle a,b \rangle$ be the free group on two generators and let $\varphi$ be the automorphism given by $\varphi(a) = ba$ and $\varphi(b)=b.$ Then the semidirect product $(\FF,\FF^+)\rtimes_\varphi(\ZZ,\NN)$ is weakly quasi-lattice ordered. To see this, let~$(p,m),(q,n)\in P=\FF^+\rtimes_{\varphi}\NN$ with an upper bound in~$P$. This implies $p\vee_{\FF^+}q <\infty$ and we can assume, say, $p\leq q$. Also, there is $r\in \FF^+$ such that $p^{-1}r=p^{-1}qq^{-1}r\in\varphi_m(\FF^+)$. Hence $p^{-1}q$ must have reduced form $$b^{i_0}ab^{i_1}a\cdots ab^{i_k},\qquad(k\geq 0)$$where $i_k\in\NN$ and $i_j\geq m$ for all~$0\leq j\leq k-1$. From this we deduce that~$p^{-1}q$ itself lies in $\varphi_m(\FF^+)$. Thus $(p,m)\vee_P(q,n)=(q,m\vee n)$ and this shows that~$(\FF,\FF^+)\rtimes_\varphi(\ZZ,\NN)$ is weakly quasi-lattice ordered as claimed.

We will now see that the canonical surjection $$\mu\colon \FF \rtimes_\varphi\ZZ\twoheadrightarrow \ZZ$$ is controlled. By Proposition \ref{semi-controlled}, $\mu(x\vee_Py)=\mu(x)\vee_{\ZZ}\mu(y)$ holds for all~$x,y\in P$ such that~$x\vee_P y<\infty$. So we need to establish condition \eqref{cm-2-new} of Definition~\ref{def-controlled}. Let $n\in\NN$. If~$n=0$, then $\mu^{-1}(0)\cap P=\FF^+\times\{0\}$ and we set~$\Sigma_0\coloneqq\{(e_{\FF^+},0)\}$. That is, $\Sigma_0$ is the unit set given by the identity of~$\FF^+\rtimes_\varphi\NN$. Suppose~$n>0$. We let 
$$\Sigma_n\coloneqq\{(\sigma,n)\in\mu^{-1}(n)\cap P: \sigma =\tau a, \tau\in\FF^+ \text{ and }b^n \text{ is not a suffix of }\tau\}.$$

We prove that $\sigma_1=\sigma_2$ whenever $(\sigma_1,n),(\sigma_2,n)\in\Sigma_n$ and $(\sigma_1,n)\vee_P(\sigma_2,n)<\infty$. Indeed, $(\sigma_1,n)\vee_P(\sigma_2,n)<\infty$ entails $\sigma_1\vee_{\FF^+}\sigma_2<\infty$. So we may assume that $\sigma_1$ is a prefix of~$\sigma_2$. Also, there exists $r\in\FF^+$ such that $\sigma_1^{-1}\sigma_2$ is a prefix of $\varphi_n(r)$. In particular, $\sigma_1^{-1}\sigma_2\neq e_{\FF^+}$ would imply that $\sigma_1^{-1}\sigma_2$ ends in~$b^na$ because~$a$ is a suffix of~$\sigma_2$. Since $\sigma_1^{-1}\sigma_2\in\FF^+$ is also a suffix of~$\sigma_2$, we must have $\sigma_1=\sigma_2$, as desired. 

Now let~$(x,n)\in P$. Write $x=\sigma u$, where $(\sigma, n)\in\Sigma_n$ and $u\in\varphi_n(\FF^+)$. More precisely, $u$ is the tail of~$x$ in~$\varphi_n(\FF^+)$. Thus $(x,n)\geq_P(\sigma, n)$ and this completes the proof that $\mu$ is controlled. 

Finally, because $\ker\mu \cap P=\FF^+\times\{0\}$, Theorem~\ref{T:nuclear} tells us that $\Cst(\FF\rtimes_{\varphi}\ZZ, \FF^+\rtimes_{\varphi}\NN)$ is also nuclear.
\end{example}

\subsection{Graph products}

  Here we briefly introduce notation from \cite{CL1} that we need. Let $\Gamma$ denote a graph with vertex set $\Lambda$ and edge set $E(\Gamma)=\{\{I,J\}: I,J\in\Lambda, I\neq J\}$. If $\{I, J\}\in E(\Gamma)$, we say that $I$ and $J$ are \emph{adjacent}.  Let $\{G_I\}_{I\in\Lambda}$ be a family of groups. Then the graph product 
\[
G\coloneqq\Gamma_{I\in\Lambda} G_I
\]
is the quotient of the free product $*_{\Lambda} G_I$ by the smallest normal subgroup containing $x_1x_2x_1^{-1}x_2^{-2}$ for all pairs $x_1\in G_I$ and $x_2\in G_J$ where $I$ and $J$ are adjacent.  A generating set for $G$ is $\sqcup_{I\in\Lambda} G_I\setminus\{e\}$. Given a generator $x$, we write $I(x)$ for the unique vertex $I$ such that $x\in G_I$. 

An \emph{expression} for an element of $x\in G$ is a word $x_1x_2\dots x_l$ in the generators which equals $x$. The graph product relations allow modification of an expression by replacing a subexpression $x_ix_{i+1}$ with $x_{i+1}x_i$ if $I(x_i)$ is adjacent to $I(x_{i+1})$; this replacement is called a \emph{shuffle}. If an expression contains a subexpression $x_ix_{i+1}$ with $I(x_i)=I(x_{i+1})$, then we obtain a shorter expression for $x$ by \emph{amalgamating} the subexpression $x_ix_{i+1}$ of length two into the subexpression $\hat x_i=x_ix_{i+1}$ of length one.  An expression is called \emph{reduced} if its length cannot be reduced by finitely many shuffles  and amalgamation. The \emph{length} $l(x)$ of $x$ is the length of any reduced expression equal to $x$.  An \emph{initial vertex} on $x$ is a vertex $I\in\Lambda$ such that   $x_i\in G_I$ is a generator in a reduced expression $x_1\dots x_{l(x)}$ for $x$  and $I(x_i)$ is adjacent to $I(x_j)$ for $j<i$. If $I$ is an initial vertex on $x$ corresponding to the generator $x_i$, we write $x_I$ for $x_i$; if $I$ is not an initial vertex on $x$ we set $x_I=e$.

Now suppose that each $G_I$ is partially ordered with positive cone $P_I$.   A reduced expression $x_1 x_2\dots x_{l(x)}$ for $x$ is \emph{positive} if $x_i\in P_{I(x_i)}$ for $1\leq i\leq l$. We say that $x$ is \emph{positive} if it has a reduced expression which is positive (and then all reduced expressions of $x$ are positive). We let $P$ denote the subsemigroup of $G$ consisting of positive elements of $G$.  Then $(G,P)$ is a partially ordered group.

Theorem~10 of \cite{CL1} says that a graph product of quasi-lattice ordered groups is a quasi-lattice ordered group, and then \cite[Proposition~13]{CL1} finds an effective algorithm to compute a least upper bound when it exists.  We now prove similar results for weakly quasi-lattice ordered groups. Our  proof is very different to those  in  \cite{CL1}. The proof of  \cite[Theorem~10]{CL1} used the characterisation of quasi-lattice ordered group from item (iv) of \cite[Lemma~7]{CL1} which does not apply to a weak quasi-lattice. The algorithm of \cite[Proposition~13]{CL1} is based on the equality \eqref{eq-algorithm} below, and to prove it \cite[Theorem~10]{CL1} is used. 

\begin{thm}\label{thm graph} A graph product $(G,P)=\Gamma_{I\in\Lambda} (G_I, P_I)$ of weakly quasi-lattice ordered groups is a weakly quasi-lattice ordered group. In particular, if $x,y\in P$ have a common upper bound in $P$ and  $I$ is any vertex in $\Lambda$, then with $x=x_Ix'$ and $y=y_Iy'$ we have
\begin{equation}\label{eq-algorithm}
x\vee y=(x_I\vee y_I)(x'\vee y'),
\end{equation}
and the vertices on $x'\vee y'$ are vertices on $x'$ and $y'$. 
\end{thm}

\begin{proof}  We start by proving that  $(G,P)$ is weakly quasi-lattice ordered.  For $n\in\NN$ we set
\[
P_n=\{(x,y)\in P\times P \colon l(x)+l(y)=n\text{\ and\ } x, y\text{\ have a common upper bound}\}.
\]
Notice that if $(x,y)\in P_0$, then $x=y=e$ and $x\vee y=e$; trivially, all the vertices on $x\vee y$ are   vertices on $x$ and $y$. 
We will prove the  result  by induction on $n$. For $n\geq  0$, assume that if $(x,y)\in P_n$, then $x\vee y$ exists and  the vertices on $x\vee y$ are the vertices on $x$ and $y$. 

Let $(x,y)\in P_{n+1}$.  Then  there exists a common upper bound $z\in P$ for $x$ and $y$.  Let $I\in\Lambda$   be  an initial vertex on $x$ or $y$.  We write $x=x_Ix'$, $y=y_Iy'$ and $z=z_Iz'$.   \cite[Lemma~11]{CL1} applies to a graph product of partially ordered groups; applying it to $x\leq z$ gives
\begin{enumerate}
\item[(i)] $x_I\leq z_I$,  and
\item[(ii)] either $x_I= z_I$ or  the vertices on $x'$ are adjacent to $I$. 
\end{enumerate}
It follows  that $x_I\leq z_I$ and $y_I\leq z_I$. Since $(G_I, P_I)$ is a weakly quasi-lattice ordered group we have $x_I\vee y_I\leq z_I$.

We claim that $z'$ is a common upper bound for $x'$ and $y'$. If $x_I=z_I$, then 
\[
x\leq z\Longrightarrow x_Ix'\leq  z_Iz'\Longrightarrow x'\leq z'
\]
 using left invariance.

Next suppose that $x_I<z_I$. Then the vertices on $x'$ are adjacent to $I$. In particular, $I$ is not a vertex on $x'$, as otherwise we could amalgamate and thus shorten the expression.  Since $x\leq z$ we have 
\[
 x^{-1}z=(x')^{-1}x_I^{-1}z_Iz'=x_I^{-1}z_I (x')^{-1}z'\in P.
\]
There exists $v_I\in P_I$ and $v'$ in $P$ such that $v:=x^{-1}z=v_Iv'$. Let $w_1\dots w_{l(w)}$ be a reduced expression for $(x')^{-1}z'$. We now want to compare the two expressions
\[
x_I^{-1}z_I w_1\dots w_{l(w)}=v_Iv'
\]
for $x^{-1}z$. For $1\leq i\leq l(w)$, if $I(w_i)=I$, then $w_i$ is in $P_I$ because $I$ is not a vertex on $(x')^{-1}$. Equating $(x^{-1}z)_I$ in both expressions, we see that $x_I^{-1}z_I\leq v_I$. Now $w_1\dots w_{l(w)}=(x_I^{-1}z_I)^{-1}v_Iv'\in P$.  Thus $(x')^{-1}z'\in P$ and $x'\leq z'$, as claimed.  

It follows from a similar argument that also $y'\leq z'$. 
 
Since $I$ is an initial vertex on $x$ or $y$ we have $l(x')+l(y')<l(x)+l(y)$.
By the induction hypothesis, $x'\vee y'$ exists and the vertices on $x'\vee y'$ are the vertices on $x'$ and $y'$. 

We claim that  the right-hand-side 
\begin{equation*}
(x_I\vee y_I)(x'\vee y')
\end{equation*}
of \eqref{eq-algorithm}
is a least upper bound for $x$ and $y$.  We first show that $x,y\leq (x_I\vee y_I)(x'\vee y')$.  We  again consider the two cases arising from  \cite[Lemma~11]{CL1}: firstly, if $x_I=z_I$, then $x_I\vee y_I=z_I=x_I$. Secondly, if $x_I<z_I$ then $I$ is adjacent to the vertices of $x'$. Thus
\begin{align*}
x^{-1}(x_I\vee y_I)(x'\vee y')&=(x')^{-1}x_I^{-1}(x_I\vee y_I)(x'\vee y')\\
&=\begin{cases}(x')^{-1}(x'\vee y')&\text{if $x_I=z_I$}\\
x_I^{-1}(x_I\vee y_I)(x')^{-1}(x'\vee y')&\text{if $x_I<z_I$},
\end{cases}
\end{align*}
which in both cases is in $P$.  Thus $x\leq (x_I\vee y_I)(x'\vee y')$, and similarly $y\leq (x_I\vee y_I)(x'\vee y')$. 
Thus $(x_I\vee y_I)(x'\vee y')$ is a common upper bound for $x$ and $y$, as claimed.

Next, suppose that $w\in P$ is a common upper bound for $x$ and $y$. We need to show that  $w\geq (x_I\vee y_I)(x'\vee y')$.  Write $w=w_Iw'$.  Since $x\leq w$, by  \cite[Lemma~11]{CL1}, either $x_I=w_I$ or $I$ is adjacent to every vertex on $x'$, and similarly for $y$. As above, we have   $x', y'\leq w'$. 
 
 First, suppose that $x_I=w_I$ or $y_I=w_I$. Then 
\begin{align*}
\big((x_I\vee y_I)(x'\vee y')\big)^{-1}w&=(x'\vee y')^{-1}(x_I\vee y_I)^{-1}w_Iw'\\
&=(x'\vee y')^{-1}w_I^{-1}w_Iw'=(x'\vee y')^{-1}w'\in P.
\end{align*}
Thus  $(x_I\vee y_I)(x'\vee y')\leq w$ as needed.

Second, suppose that  $x_I<w_I$ and  $y_I< w_I$. Recall that  the vertices  on $x'$ and $y'$ are adjacent to $I$, and hence so are the vertices on  $x'\vee y'$ and $(x'\vee y')^{-1}$. Thus 
\begin{align*}
\big((x_I\vee y_I)(x'\vee y')\big)^{-1}w
&=(x'\vee y')^{-1}\big((x_I\vee y_I)^{-1}w_I\big)w'\\\
&=\big((x_I\vee y_I)^{-1}w_I\big)\big((x'\vee y')^{-1}w'\big);
\end{align*}
this is in $P$ because both $(x_I\vee y_I)^{-1}w_I$ and $(x'\vee y')^{-1}w'$ are in $P$. Thus $(x_I\vee y_I)(x'\vee y')\leq w$ and hence is a least upper bound for $x$ and $y$, as claimed. 
We have proved that $(G,P)$ is a weak quasi-lattice, and that if $I$ is an initial vertex on $x$ or $y$, then \eqref{eq-algorithm} holds.
On the other hand,  if $I$ is not an initial vertex of $x$ or $y$, then $x_I=e=y_I$, and  \eqref{eq-algorithm} holds trivially.
\end{proof}

Let  $\Gamma_{I\in\Lambda} (G_I, P_I)$ be a graph product of partially ordered groups $(G_I, P_I)$. By the universal property of the free product there is a unique group homomorphism $\phi: *_\Lambda G_I \to \oplus_{I\in\Lambda}G_I$ extending the inclusion of $G_I$ in $*_\Lambda G_I$.  Since $\phi(G_I)$ and $\phi(G_J)$ commute if $I\neq J$, it follows that $\phi$ factors through the quotient $\Gamma_{I\in\Lambda} G_I$. Thus we have a  homomorphism
\begin{equation}\label{defn-phi}
\phi: \Gamma_{I\in\Lambda} (G_I, P_I)\to\bigoplus_{I\in\Lambda}(G_I, P_I),
\end{equation}
also denoted by $\phi$,  which extends the inclusion of $G_I$ in $\Gamma_{I\in\Lambda} G_I$.

 Let $I\in\Lambda$ and for $x\in P$ write $x=x_Ix'$. Then $\phi(x)=\phi(x_I)\phi(x')=x_I\phi(x')$ and $\phi(x)_I=x_I\phi(x')_I$. 

Now suppose that $(G,P)=\Gamma_{I\in\Lambda} (G_I, P_I)$ is  a graph product of partially  ordered groups $(G_I, P_I)$. It is proved in \cite[Proposition~19]{CL1} that if $(G,P)$ is  a quasi-lattice ordered group, then $\phi$ is a controlled map with $\ker\phi\cap P=\{e\}$. We now prove the analogous result when $(G,P)$ is only a weakly quasi-lattice ordered group.

\begin{prop}\label{prop graph controlled map}  Let $(G,P)=\Gamma_{I\in\Lambda} (G_I, P_I)$ be a graph product of  weakly quasi-lattice ordered groups. Let $\phi: \Gamma_{I\in\Lambda} (G_I, P_I)\to\bigoplus_{I\in\Lambda}(G_I, P_I)$ be the homomorphism defined at \eqref{defn-phi}.  Let $x,y\in P$ such that $x\vee y<\infty$. Then
\begin{enumerate}
\item\label{prop-controlled-a} $\phi(x)\vee\phi(y)=\phi(x\vee y)$, and 
\item\label{prop-controlled-b}  $\phi(x)=\phi(y)\Longrightarrow x=y$. 
\end{enumerate}
In particular, $\phi$ is a controlled map.
\end{prop}

\begin{proof} We first observe that $\phi$ is order-preserving. Suppose that $g,h\in G$ with $g\leq h$. Then $h^{-1}g\in P$ and $\phi(h)^{-1}\phi(g)=\phi(h^{-1}g)\subseteq\phi(P)=\oplus P_I$. Thus $\phi(g)\leq \phi(h)$.

For  (\ref{prop-controlled-a}) we induct on $l(x)+l(y)$. When $l(x)+l(y) = 0$ we have $x=y=e$ and $\phi(x)\vee\phi(y)=e=\phi(e)=\phi(x\vee y)$.  Let $n\geq 0$ and  assume that if $v, w\in  P$ with $v\vee w<\infty$ and $l(v)+l(w)\leq n$, then $\phi(v)\vee\phi(w)=\phi(v\vee w)$. 

Suppose that $x,y\in P$ with $x\vee y<\infty$ and $l(x)+l(y)=n+1$. Let $I\in\Lambda$ be an initial vertex on $x$. Write $x=x_Ix'$ and $y=y_Iy'$, and notice that $l(x')<l(x)$. 
We  again consider the cases arising from  \cite[Lemma~11]{CL1}. First, suppose that $x_I<x_I\vee y_I$ and that $y_I<x_I\vee y_I$. Then the vertices on $x'$ and $y'$ are adjacent to $I$.  We have 
\begin{align*}
\phi(x\vee y)&=\phi\big((x_I\vee y_I)(x'\vee y')\big) \quad\text{(by Theorem~\ref{thm graph})}\\
&=\phi(x_I\vee y_I)\phi(x'\vee y')\\
&=\phi(x_I\vee y_I)\big(\phi(x')\vee \phi(y')\big)\quad\text{(by the induction hypothesis)}\\
&=\big(\phi(x_I)\vee \phi(y_I)\big)\big(\phi(x')\vee \phi(y')\big).\\
\intertext{
Since the vertices on $\phi(x')$ and $\phi(y')$ are the vertices on $x'$ and $y'$, and $I$ is adjacent to these, we get}
\phi(x\vee y)&=\big(\phi(x_I)\vee \phi(y_I)\big)\vee \big(\phi(x')\vee \phi(y')\big)\\
&=\big(\phi(x_I)\vee \phi(x')\big)\vee  \big(\phi(y_I)\vee  \phi(y')\big)\\
&=\phi(x_I)  \phi(x') \vee \phi(y_I)  \phi(y')\\
&=\phi(x)\vee\phi(y). 
\end{align*}

Second, suppose that $x_I=x_I\vee y_I$ or that $y_I=x_I\vee y_I$. We assume, without loss of generality, that  $x_I=x_I\vee y_I$. 
If $y_I\neq e$, then $l(y')<l(y)$ and 
\begin{align*}
\phi(x\vee y)&=\phi(x_Ix'\vee  y_Iy')\\
&=\phi\big(y_I(y_I^{-1}x_Ix'\vee y')\big)\quad\text{(by left invariance)}\\
&=\phi(y_I)\phi(y_I^{-1}x_Ix'\vee y')\\
&=\phi(y_I)\big(\phi(y_I^{-1}x_Ix')\vee \phi(y')\big) \quad\text{(by the induction hypothesis)}\\
&=\phi(x_Ix')\vee \phi(y_Iy')\\
&=\phi(x)\vee\phi(y).
\end{align*}
 If $y_I=e$, then $y=y'$. It follows from \cite[Lemma~11]{CL1} applied to $y\leq x\vee y$ that   $I$ is adjacent to the vertices of $y$. Now
\begin{align*}
\phi(x\vee y)&=\phi\big((x_I\vee y_I)(x'\vee y')\big)\\
&=\phi\big((x_I )(x'\vee y)\big)\\
&=\phi(x_I)\big(\phi(x')\vee\phi(y)\big) \quad\text{(by the induction hypothesis, since $l(x')<l(x)$)}\\
&=\phi(x_I) \phi(x')\vee\phi(x_I)\phi(y)\\
&=\phi(x)\vee\phi(x_I)\vee\phi(y)\\
&=\phi(x)\vee\phi(y)
\end{align*}
since $\phi(x_I)\leq \phi(x)$. 
We have now proved (\ref{prop-controlled-a}). 

For (\ref{prop-controlled-b}), suppose that $\phi(x)=\phi(y)$.  Write $x\vee y=xu=yv$ where $u, v\in P$.  By~(\ref{prop-controlled-a}) we have
\[\phi(x)=\phi(x)\vee\phi(y)=\phi(x\vee y)=\phi(x)\phi(u).
\]
Thus $\phi(u)=e$. But  $u_I\leq\phi(u)_I$ for all $I\in\Lambda$. So $\phi(u)=e$ implies $u=e$. Similarly, $v=e$ and hence $x=y$.  This gives~(\ref{prop-controlled-b}), and it follows that $\phi$ is a controlled map. 
\end{proof}

\begin{cor}
 Let $(G,P)=\Gamma_{I\in\Lambda} (G_I, P_I)$ be a graph product of  weakly quasi-lattice ordered groups.  Suppose that  $G_I$ is an amenable  group for every $I\in\Lambda$. Then
 $(G,P)$ is  amenable as a weakly quasi-lattice ordered group and   $\Cst(G, P)$ is nuclear. 
\end{cor}

\begin{proof} Since each $G_I$ is amenable, so is the direct sum $\oplus G_I$. Now by Proposition~\ref{prop graph controlled map}, $\phi$ is a controlled map into an amenable group. Since $\ker\phi\cap P=\{e\}$, $\Cst(G, P)$ is nuclear and $(G,P)$ is amenable by Theorem~\ref{T:nuclear}. 
\end{proof}

\section{Amenability}\label{sec-amenability}
Recall from \S2 that $(G,P)$ is amenable in the sense of Nica if the Toeplitz representation $T:P\to B(\ell^2(P))$ induces an isomorphism $\pi_T: \Cst(G,P)\to \Tt(G, P)$.    We can now extend the amenability theorem of \cite{aHRT} to weak quasi-lattices admitting a controlled map in the sense of Definition~\ref{defn-less-controlled}. The proof ideas are the ones from \cite[Theorem~3.2]{aHRT} adjusted to the new definition. Since nuclearity of $\Cst(G,P)$ implies that $(G,P)$ is amenable by  \cite[Theorem~6.42]{Li-book}, Theorem~\ref{thm-amenability} is of interest only when we don't already know that $\Cst(G,P)$ is nuclear.

\begin{thm}\label{thm-amenability}  Let $(G,P)$ and $(K,Q)$ be weakly quasi-lattice ordered groups.
Suppose that $\mu\colon(G,P)\to (K,Q)$ is a controlled map in the sense of Definition~\textup{\ref{defn-less-controlled}}, that $K$ is amenable and that $(\ker\mu, \ker\mu\cap P)$ is amenable. Then  $(G,P)$ is amenable. 
\end{thm}

For $k\in Q$ we consider the subspaces
\[
H_k\coloneqq\clsp\{e_{s_n^\lambda \alpha}\colon \lambda\in\Lambda_k, n\in\NN, \alpha\in\ker\mu\cap P\}.
\]
Notice that $H_k=\ell^2(\mu^{-1}(k)\cap P)$, and in particular, that $H_e=\ell^2(\ker\mu\cap P)$.

\begin{lem}\label{lem-isometric} Let $(G,P)$ and $(K,Q)$ be weakly quasi-lattice ordered groups.
Suppose that $\mu\colon(G,P)\to (K,Q)$ is a controlled map in the sense of Definition~\textup{\ref{defn-less-controlled}}. Suppose that $(\ker\mu,\ker\mu\cap P)$ is amenable. 
Then $\pi_T(\cdot)|_{H_k}$ is isometric on $B_k$ for all  $k\in Q$. 
\end{lem}

\begin{proof} We start with $k=e$. 
Let $\alpha, \beta\in \ker\mu\cap P$. Then $\mu(\alpha\beta)=e=\mu(\alpha^{-1}\beta)$. Thus $T_\alpha|: H_e\to H_e$ with $T_\alpha \epsilon_\beta =\epsilon_{\alpha\beta}=S_\alpha \epsilon_\beta$ and $T_\alpha^*\epsilon_\beta=S_\alpha^*\epsilon_\beta$. Thus
\[
\pi_T(\cdot)|_{H_e}: B_e\to \Tt(\ker\mu, \ker\mu\cap P)
\]
is well-defined. We write $S$ for the Toeplitz representation of $\ker\mu\cap P$ on $H_e=\ell^2(\ker\mu\cap P)$.  Since $\big(\ker\mu,\ker\mu\cap P\big)$ is amenable, $\pi_S$ is faithful, and  the restriction of $w$ to $\ker\mu\cap P$ induces an isomorphism $
\pi_{w|}: \Cst(\ker\mu, \ker\mu \cap P)\to B_e$ by  Lemma~\ref{lem-subalgebra}(\ref{lem-subalgebra-a}).
We have $\pi_S=\pi_T(\cdot)|_{H_e}\circ \pi_{w|}$. Now $\pi_T(\cdot)|_{H_e}$ has to be faithful on $B_e$ as well.

\medskip

Next we consider $k\neq e$.  By Proposition~\ref{prop-structure-fpa}, $B_k=\varinjlim B_{k,n}$ and $B_{k,n}=\varinjlim B_{k,n,F}$. Thus it suffices to show that $\pi_T(\cdot)|_{H_k}$ is isometric on $B_{k,n,F}$ for $n\in\NN$ and finite subsets $F$ of $\Lambda_k$.

We start by showing that $H_k$ is invariant for $\pi_T(B_{k,n,F})$ so that $\pi_T(\cdot)|_{H_k}$  makes sense. 
We have 
\begin{align}
B_{k,n,F}&=\lsp\{w_{s^\lambda_n} Dw^*_{s^{\rho}_n}\colon \lambda,\rho\in F \text{\ and\ }D\in B_e\}\label{no-closure}\\
&=\clsp\{w_{s^\lambda_n\alpha}   w^*_{s^{\rho}_n\beta}\colon \lambda,\rho\in F\text{ and }\alpha,\beta\in \ker\mu\cap P\}.
\notag
\end{align}
For a spanning element $w_{s^\lambda_n\alpha}   w^*_{s^{\rho}_n\beta}$ of $B_{k,n,F}$ and a spanning element $e_{s_m^\sigma \gamma}$ of $H_k$ we have
\[
\pi_T(w_{s^\lambda_n\alpha}   w^*_{s^{\rho}_n\beta})e_{s_m^\sigma \gamma}
=\begin{cases}
e_{s^\lambda_n\alpha(s^{\rho}_n\beta)^{-1}s_m^\sigma \gamma}
&\text{if $s^{\rho}_n\beta\leq s_m^\sigma \gamma$ (and then $\rho=\sigma$)}\\
0&\text{else.}
\end{cases}
\]
Since $\mu\big(\alpha(s^{\rho}_n\beta)^{-1}s_m^\sigma \gamma\big)=e$ we see that $\pi_T(w_{s^\lambda_n\alpha}   w^*_{s^{\rho}_n\beta})e_{s_m^\sigma \gamma}\in H_k$. It follows that $H_k$ is invariant for $\pi_T(\cdot)$.

Now suppose that $a\in B_{k,n,F}$ and $\pi_T(a)|_{H_k}=0$. Then $a=\sum_{\lambda,\rho\in F}w_{s_n^\lambda} a_{\lambda,\rho}w_{s_n^\rho}$ for some $a_{\lambda,\rho}\in B_e$. (Here we use the property that there is no closure at \eqref{no-closure}.)
Fix $\eta, \xi\in F$. Then
\[
T_{s_n^\eta}^*\pi_T(a)T_{s_n^\xi}=\pi_T(w_{s_n^\eta}^*aw_{s_n^\xi})=\pi_T(a_{\eta,\xi}).
\]
Since $T_{s_n^\xi}$ is an isometry from $H_e$ to $H_k$ and $\pi_T(a)|_{H_k}=0$, it follows that $\pi_T(a)T_{s_n^\xi}|_{H_e}=0$.
Thus $\pi_T(a_{\eta,\xi})|_{H_e}=0$. But we proved that $\pi_T(\cdot)|_{H_e}$ is faithful on $B_e$ above.
 Thus $a_{\eta,\xi}=0$. It follows that $a=0$. Thus $\pi_T(\cdot)|_{H_k}$ is faithful on $B_{k,n,F}$, and hence is isometric on $B_{k,n, F}$.  
\end{proof}

\begin{prop}\label{prop-Toeplitz-rep-faithful} Let $(G,P)$ and $(K,Q)$ be weakly quasi-lattice ordered groups.
Suppose that $\mu\colon(G,P)\to (K,Q)$ is a controlled map in the sense of Definition~\textup{\ref{defn-less-controlled}}. If $(\ker\mu,\ker\mu\cap P)$ is amenable, then $\pi_T$ is faithful on $\Cst(G,P)^{\delta_\mu}$.
\end{prop}

\begin{proof} By Proposition~\ref{prop-structure-fpa}, $\Cst(G,P)^{\delta_\mu}=\varinjlim_{I\in\Ii}C_I$. So it suffices to show that $\pi_T$ is isometric on each $C_I$. Let $a\in C_I$ such that $\pi_T(a)=0$. Since $C_I=\sum_{k\in I} B_k$, there exist $a_k\in B_k$ such that $a=\sum_{k\in I} a_k$. Then $0=\sum_{k\in I}\pi_T(a_k)$.

We claim that $k\not\leq l$ implies that $\pi_T(B_k)H_l=\{0\}$. We prove the contrapositive. Suppose that $\pi(B_k)H_l\neq \{0\}$. Then there exists $q\in \mu^{-1}(k)\cap P$ and a spanning vector $e_{s_n^\lambda \alpha}\in H_l$ such that $\pi_T(w_q^*)e_{s_n^\lambda \alpha}\neq 0$.  We have 
\[
\pi_T(w_q^*)e_{s_n^\lambda \alpha}=
\begin{cases}
e_{q^{-1}s_n^\lambda \alpha}&\text{if $q\leq s_n^\lambda \alpha$}\\
0&\text{else.}
\end{cases}
\]
Here $s_n^\lambda\in \mu^{-1}(l)\cap P$ and $\alpha\in\ker\mu\cap P$.
So $\pi_T(B_k)H_l\neq \{0\}$ implies  $k=\mu(q)\leq\mu (s_n^\lambda \alpha)=\mu (s_n^\lambda)=l$. This proves the claim.

Now let $l_1$ be a minimal element in $I$. Then
\[0=\big(\sum_{k\in I}\pi_T(a_k)\big)H_{l_1}=\pi_T(a_{l_1})H_{l_1}.
\]
Since $(\ker\mu,\ker\mu\cap P)$ is amenable, Lemma~\ref{lem-isometric} implies that $\pi_T(\cdot)|_{H_{l_1}}$ is isometric on $B_{l_1}$. Thus $a_{l_1}=0$. 
Now let $l_2$ be a minimal element of $I\setminus\{l_1\}$. Repeat the above argument
to get $a_{l_2} = 0$. Since $I$ is finite, we conclude that $a=0$. Thus $\pi_T$ is faithful on $C_I$, and hence is isometric on $C_I$. It follows that $\pi_T$ is isometric on $\Cst(G,P)^{\delta_\mu}$. 
\end{proof}

\begin{proof}[Proof of Theorem~\ref{thm-amenability}] We need to show that the conditional expectation \[E:\Cst(G,P)\to\clsp\{w_pw_p^*\colon p\in P\}\] is faithful. Since $K$ is amenable, the conditional expectation $\Psi_\mu\colon\Cst(G,P)\to \Cst(G,P)^{\delta_\mu}$ is faithful by \cite[Lemma~3.5]{aHRT}. We claim that it suffices to show that $E$ is faithful on $\Cst(G,P)^{\delta_\mu}$. To see this, let $a\in \Cst(G,P)$ and suppose that $E(a^*a)=0$. We observe that $E=E\circ \Psi_\mu$. Since $\Psi_\mu$ is positive, there exists $b\in \Cst(G,P)^{\delta_\mu}$ such that  $\Psi_\mu(a^*a)=b^*b$. Then
$0=E(a^*a)=E\circ \Psi_\mu(a^*a)=E(b^*b)$.
If $E$ is faithful on $\Cst(G,P)^{\delta_\mu}$, we get $b=0$. Then $\Psi_\mu(a^*a)=0$ implies $a=0$ because $\Psi_\mu$ is faithful. This proves the claim. 

Let $P_p$ be the orthogonal projection onto the subspace  $\lsp\{e_p\}$. There is a conditional expectation $\Delta$ on $B(\ell^2(P))$ such that $\Delta(T)=\sum_{p\in P}P_pTP_p$. It is easy to see that $\Delta$ is faithful and, by computing on generators, that
\[
\Delta\circ\pi_T=\pi_T\circ E.
\]
Now suppose that $b\in \Cst(G,P)^{\delta_\mu}$ such that $E(b^*b)=0$.  Then
$0=\pi_T(E(b^*b))=\Delta(\pi_T(b^*b))$. Since $\Delta$ is faithful, $\pi_T(b^*b)=0$. By Proposition~\ref{prop-Toeplitz-rep-faithful} the Toeplitz representation is faithful on $\Cst(G,P)^{\delta_\mu}$, and hence $b^*b=0$. Thus $b=0$, and $E$ is faithful on $\Cst(G,P)^{\delta_\mu}$. As mentioned above, this implies that $E$ is faithful.  
\end{proof}

\appendix
\section{Characterising nuclearity using the conditional expectation}

\begin{prop}\label{prop-iff-diagonal}
Let $(G,P)$ be a weakly quasi-lattice ordered group.   Then $\Cst(G, P)$ is nuclear if and only if, for every unital $\Cst$\nb-algebra $A$, the conditional  expectation $E_{A,\max}:A\otimes_\maxt \Cst(G, P)\to A\otimes \clsp\{w_p w_p^*:p\in P\}$ is faithful.
\end{prop}

For the proof of Proposition~\ref{prop-iff-diagonal} we need the following lemma. 

\begin{lem}\label{lem-expectation} Let $A$ and  $C$ be $\Cst$\nb-algebras, and let $B$ be a $\Cst$\nb-subalgebra of $C$.  Suppose that  $E:C\to B$ is a conditional expectation. Then $\id\odot E: A\odot C\to A\odot B$  extends to  conditional expectations 
 \[
E_{A, \min}:A\otimes_\mint C\to A\otimes_\mint B\text{\ and \ } E_{A, \max}:A\otimes_\maxt C\to A\otimes_\maxt B.
\] 
If $E$ is faithful, then so is $E_{A, \min}$. 
\end{lem}

\begin{proof}
Since $E$ is a conditional expectation, it is a completely positive contractive map by \cite[Theorem~1.5.10]{B-O}. Now both  $\id$ and $E$ are completely positive maps, and so by \cite[Theorem~3.5.3]{B-O}, $\id\odot E$ extends to completely positive maps $E_{A,\min}$ and  $E_{A,\max}$ of norm $1$. Then $E_{A,\min}$  and  $E_{A,\max}$ are linear idempotents of norm $1$ that act as the identity map on $A\otimes_\mint B$ and $A\otimes_\maxt B$, respectively, and hence are conditional expectations.

Now suppose that $E$ is faithful. We will show that $E_{A, \min}$ is faithful. Let $X$ be the standard $A$-$A$ imprimitivity bimodule. Let $Y\coloneqq C$ be the right $B$-module with action $y\cdot b\coloneqq yb$ for all $y\in Y$ and $b\in B$. Since $E$ is faithful,  $\langle y_1\,,\, y_2\rangle \coloneqq E(y_1^*y_2)$ is a $B$-valued inner product and turns $Y$ into a right Hilbert $B$-module.  View $Y$ as a  $\Kk(Y)$-$B$ imprimitivity bimodule. By \cite[Proposition~3.36]{tfb} $X\odot Y$ completes to give an $A\otimes\Kk(Y)$-$A\otimes_\mint B$ imprimitivity bimodule with right inner product characterised by
$
\langle x_1\otimes y_1\,,\,  x_2\otimes y_2\rangle = \langle x_1\,,\, x_2\rangle \otimes \langle y_1\,,\, y_2\rangle 
$
for $x_1, x_2\in A$ and $y_1, y_2\in C$. In particular,
\begin{align*}
\langle x_1\otimes y_1\,,\,  x_2\otimes y_2\rangle =x_1^*x_2\otimes E(y_1^*y_2)&=\id\otimes E(x_1^*x_2\otimes y_1^*y_2)\\
&= E_{A,\mathrm{min}}(x_1^*x_2\otimes y_1^*y_2).
\end{align*}
It follows that for $z\in A\otimes_\mint C$ we have 
$
\langle z\,,\,  z\rangle =E_{A,\min}(z^*z).
$
Thus $E_{A,\min}(z^*z)=0$ implies $z=0$, and $E_{A,\min}$ is faithful.
\end{proof}

\begin{proof}[Proof of Proposition~\ref{prop-iff-diagonal}] 
Let  $A$ be a unital  $\Cst$\nb-algebra and write $q_A:A\otimes_\maxt \Cst(P)\to A\otimes_\mint \Cst(P)$ for the quotient map. By Lemma~\ref{lem-expectation},   $\id\odot E$ 
extends to conditional expectations 
\begin{gather*}
E_{A,\max}:A\otimes_\maxt \Cst(G, P)\to A\otimes \clsp\{w_p w_p^*:p\in P\}\\
E_{A,\min}:A\otimes_\mint \Cst(G, P)\to A\otimes \clsp\{w_p w_p^*:p\in P\},
\end{gather*}
and $E_{A,\min}$ is faithful if $E$ is. We have   $E_{A,\max}=E_{A,\min}\circ q_A$. 

First, suppose that $\Cst(G, P)$ is nuclear and fix a unital $\Cst$\nb-algebra $A$.  Since $(G,P)$ is a weak quasi-lattice, it follows from   \cite[Theorem~6.44]{Li-book} that $(G,P)$ is amenable. Thus $E$ is faithful, and then so is $E_{A,\min}$ by Lemma~\ref{lem-expectation}. Since $\Cst(G, P)$ is nuclear,  $q_A$ is injective.   Thus $E_{A,\max}=E_{A,\min}\circ q_A$ is faithful.

Second, suppose that, for every unital $\Cst$\nb-algebra $A$, the expectation $E_{A,\max}$ is faithful. Since $E_{A,\max}=E_{A,\min}\circ q_A$ we must have that $q_A$ injective. Thus $A\otimes_\maxt \Cst(G, P)$ is isomorphic to $A\otimes_\mint \Cst(G, P)$, and $\Cst(G, P)$ is nuclear by \cite[Lemma~B.42]{tfb}. 
\end{proof}

\subsection*{Acknowledgments} The authors are  grateful to the referee for their careful reading and useful comments.

\end{document}